\documentclass[final,leqno]{siamltex}

\usepackage{amsmath,amssymb,amscd,amsxtra,amsfonts,mathrsfs}
\usepackage{epsf,graphicx,epsfig,cite,latexsym,color,hyperref}

\newcommand{\bm}[1]{\boldsymbol{#1}}

\newcommand{\norm}[1]{\left\Vert#1\right\Vert}
\newcommand{\abs}[1]{\left\vert#1\right\vert}
\newcommand{\ls}{\lesssim}

\newcommand{\eq}[1]{\begin{align}#1\end{align}}
\newcommand{\eqn}[1]{\begin{align*}#1\end{align*}}
\newcommand{\db}{\displaybreak[0]}
\newcommand{\LL}{\sqrt{L_1L_2}}

\newcommand{\al}{\alpha}
\newcommand{\be}{\beta}
\newcommand{\de}{\delta}

\newcommand{\ep}{\varepsilon}

\newcommand{\Ga}{\Gamma}

\newcommand{\na}{\nabla}

\newcommand{\Om}{\Omega}

\newcommand{\si}{\sigma}

\newcommand{\ze}{\zeta}

\title{An Adaptive Finite Element DtN Method for Maxwell's Equations in
Biperiodic Structures}

\author{Xue Jiang\thanks{School of Science, Beijing University of Posts and
Telecommunications, Beijing 100876, China. The research was supported in part
by NSFC grants 11771057, 11401040, and 11671052. ({\tt jxue@lsec.cc.ac.cn}).}
\and Peijun Li\thanks{Department of Mathematics, Purdue University, West Lafayette,
Indiana 47907, USA. ({\tt lipeijun@math.purdue.edu}).}
\and Junliang Lv\thanks{School of Mathematics, Jilin University, Changchun
130012, China. The research was partially supported by NSFC grant
11301214 and by Science Challenge Project grant TZ2016002.
({\tt lvjl@jlu.edu.cn}).} \and Zhoufeng Wang\thanks{School of Mathematics and
Statistics, Henan University of Science and Technology, Henan, 471023, China.
({\tt zfwang801003@126.com}).} \and Haijun Wu\thanks{Department of Mathematics,
Nanjing University, Jiangsu 210093, China. The research was supported in part by
NSFC grants 11525103, 91630309, and 11621101. ({\tt hjw@nju.edu.cn}).} \and
Weiying Zheng\thanks{NCMIS, LSEC, ICMSEC, Academy of Mathematics and System
Sciences, Chinese Academy of Sciences, Beijing, 100190, China. The research was
supported in part by NSFC grants 11171334 and 91430215, by the Funds for
Creative Research Groups of China (Grant No. 11021101), and by National 863
Project of China under the grant 2012AA01A309.}}

\begin{document}

\maketitle

\begin{abstract}
Consider the diffraction of an electromagnetic plane wave by a biperiodic
structure where the wave propagation is governed by the three-dimensional
Maxwell equations. Based on transparent boundary condition, the grating problem
is formulated into a boundary value problem in a bounded domain. Using a duality
argument technique, we derive an a posteriori error estimate for the finite
element method with the truncation of the nonlocal Dirichlet-to-Neumann (DtN)
boundary operator. The a posteriori error consists of both the finite element
approximation error and the truncation error of boundary operator which decays
exponentially with respect to the truncation parameter. An adaptive finite
element algorithm is developed with error controlled by the a posterior error
estimate, which determines the truncation parameter through the truncation error
and adjusts the mesh through the finite element approximation error. Numerical
experiments are presented to demonstrate the competitive behavior of the
proposed adaptive method.
\end{abstract}

\begin{keywords}
Maxwell's equations, biperiodic grating problem, adaptive finite element
method, transparent boundary condition, a posteriori error estimate
\end{keywords}

\begin{AMS}
65M30, 78A45, 35Q60
\end{AMS}

\pagestyle{myheadings}
\thispagestyle{plain}
\markboth{X. Jiang, P. Li, J. Lv, Z. Wang, H. Wu, and W. Zheng}{Adaptive
DtN Method for Maxwell's Equations}

\section{Introduction}

Consider the diffraction of a time-harmonic electromagnetic plane wave by a
biperiodic structure in $\mathbb R^3$. A biperiodic structure is also called a
doubly periodic structure, crossed grating, or two-dimensional grating in
optics. Scattering theory in periodic structures have many important
applications in micro-optics, which include the design and fabrication of
optical elements such as corrective lenses, antireflective interfaces, beam
splitters, and sensors. The basic electromagnetic theory of gratings has been
studied extensively since Rayleigh's time \cite{R07}. Recent advance has been
greatly accelerated due to the development of new approaches and numerical
methods including differential methods, integral methods, analytical
continuation, variational methods, and others. An introduction to
grating problems can be found in Petit \cite{P80}. We refer to \cite{BDC95} and
the references cited therein for the mathematical studies on well-posedness of
the diffraction grating problems. Numerical methods can be found
in \cite{B95, B97, BCL14, BCW05, BR93a, BR93b, DF92, HNS12, JL17, NS91, WL09}
for various approaches including integral equation method, finite element
method, and boundary perturbation method for solving the direct and inverse
diffraction grating problems. A comprehensive review can be found in
\cite{BCM01} on diffractive optics technology and its mathematical modeling as
well as computational methods. One may consult monographs \cite{CK83, CK98,
J93, M03, N01} for extensive accounts of integral equation methods and finite
element methods for direct and inverse electromagnetic scattering problems in
general structures.

The scattering problems are usually imposed in open domains, which need to be
truncated into bounded computational domains when applying numerical methods
such as finite element method or finite difference method. Therefore,
appropriate boundary conditions are required on the boundaries of the truncated
domains in order to avoid artificial wave reflection. These boundary conditions
are called absorbing boundary conditions (ABCs) \cite{EM77, BT80},
non-reflecting boundary conditions \cite{GK95, H99}, or transparent boundary
conditions (TBCs) \cite{GK04}. They are still the subject matter of much ongoing
research. Another effective truncation strategy is the perfectly matched layer
(PML) technique, which was first introduced by Berenger in \cite{B94}. In
computational wave propagation, it has become an active research area on
constructions and analysis of PML absorbing layers for various scattering
problems ever since then \cite{BP08, CC08, CM98, LWZ11, TC01, TY98}. The basic
idea of the PML technique is to add an artificial layer of medium to surround
the physical domain. Such medium is generally designed to make the outgoing
waves to decay exponentially, so that a homogeneous Dirichlet boundary condition
can be imposed on the exterior boundary of the layer. In order to effectively
choose the parameters of medium and the thickness of PML absorbing layer, the
adaptive finite element methods were analyzed for the diffraction grating
problems \cite{BCW05, BLW10, CW03}. The adaptive finite element PML method has
also been applied to solve the obstacle scattering problems in \cite{BW05,
CL05}.

Recently, combined with the TBC truncation approach, adaptive finite element
methods were developed to solve the two-dimensional acoustic obstacle scattering
problems \cite{JLLZ17, JLZ13} and the one-dimensional diffraction grating
problem \cite{WBLLW15}. Unlike the PML technique, the finite element TBC method
does not require extra an artificial layer of domain to surround the physical
domain. Consequently, an obvious advantage of the TBC method is that the size of
the computational domain can be remarkably reduced, since the artificial
boundary can be put as close as possible to surround the obstacle due to the
exactness of transparent boundary condition. This merit may decrease the scale
of the resulting linear system of algebraic equations. It should be pointed out
that the TBC is defined by a nonlocal Dirichlet-to-Neumann (DtN) operator, which
is given by an infinite Fourier series. In practical computation, one has to
choose a positive integer $N$ to truncate the infinite series into a sum of
finite sequence. In \cite{JLLZ17, WBLLW15}, the authors derived \emph{a
posteriori} error estimates, which are composed of the finite element
discretization error and the truncation error of the DtN operator.

The goal of this paper is to extend the finite element DtN method proposed in
\cite{JLLZ17, WBLLW15} to the two-dimensional diffraction grating
problem. The extension is non-trivial since the techniques differ
greatly from \cite{JLLZ17, WBLLW15}. We need to consider more complicated
three-dimensional Maxwell equations instead of the two-dimensional Helmholtz
equation. In this work, we derive an \emph{a posteriori} error estimate which
not only takes into account of the finite element discretization error but also
the truncation error of the boundary operator. In \cite{HNPX11}, it was shown
that the convergence could be arbitrarily slow for the truncated DtN mapping to
the original DtN mapping in its operator norm for the obstacle scattering
problem. The same issue arises for the diffraction grating problem. To overcome
this difficulty, a new duality argument is introduced for the \emph{a
posteriori} error estimate between the solution of the diffraction problem and
the finite element solution. The estimate is used to design the adaptive finite
element algorithm to choose elements for refinement and to determine the
truncation parameter $N$ in the Fourier series. We show that the truncation
error decays exponentially with respect to $N$. The numerical experiments
demonstrate a comparable behavior to the adaptive PML method developed in
\cite{BLW10}, and show much more competitive efficiency by adaptively refining
the mesh as compared with uniformly refining the mesh. This work provides a
viable alternative to the adaptive finite element method with the PML technique
for solving the diffraction grating problem. The method is expected to be
applicable to solve many other wave propagation problems in open domains and
even more general model problems where TBCs are available but the PML may not be
applied.

The paper is organized as follows. In Section 2, we introduce the model
problem of the diffraction of an electromagnetic plane wave by a bi-periodic
structure and its weak formulation by using the transparent boundary condition.
The finite element discretization with truncated DtN operator is presented in
Section 3. Section 4 is the main body of the work and is devoted to the a
posteriori error estimate by using a duality argument. In Section 5, we present
some numerical experiments to demonstrate the competitive behavior of the
proposed adaptive DtN method. The paper is concluded with some general remarks
in Section 6.

\section{Problem formulation}

In this section, we introduce the model problem of the diffraction of an
electromagnetic plane wave by a biperiodic structure, and its variational
formulation by using the transparent boundary condition.

\subsection{Maxwell's equations}

The electromagnetic fields in the whole space are governed by the
time-harmonic (time-dependence $e^{-{\rm i}\omega t}$) Maxwell's equations:
\begin{equation}\label{me}
\nabla\times \bm{E}-{\rm i}\omega\mu\bm{H}=0,\quad
\nabla\times\bm{H}+{\rm i}\omega\varepsilon\bm{E}=0,
\end{equation}
where $\bm{E}$ and $\bm{H}$ are the electric field and the magnetic field,
respectively. The physical structure is described by the dielectric permittivity
$\varepsilon(\bm{x})\in L^{\infty}(\mathbb{R}^3)$ and magnetic
permeability $\mu(\bm{x})\in L^{\infty}(\mathbb{R}^3)$,
$\bm{x}=(x_1,x_2,x_3)^\top$. The dielectric permittivity $\varepsilon$ and the
magnetic permeability $\mu$ are assumed to be periodic in the $x_1$ and $x_2$
directions with periods $L_1$ and $L_2$, respectively, i.e.,
\begin{align*}
\varepsilon(x_1+n_1L_1, x_2+n_2L_2, x_3)&=\varepsilon(x_1,x_2,x_3),\\
\mu(x_1+n_1L_1, x_2+n_2L_2, x_3)&=\mu(x_1,x_2,x_3),
\end{align*}
for $x_j\in\mathbb{R}$, where $n_1$, $n_2$ are integers. Throughout we assume
that $\text{Im}\varepsilon\geq 0$, $\text{Re}\varepsilon>0$, and $\mu>0$.
The problem
geometry is shown in Figure \ref{pg}. Let
\[
\Omega=\{\bm{x}\in\mathbb R^3: 0<x_1<L_1,\, 0<x_2<L_2,\, b_2<x_3<b_1\},
\]
where $b_j, j=1, 2$ are  constants. Denote by $\Omega_1=\{\bm{x}\in\mathbb R^3:
0<x_1<L_1,\, 0<x_2<L_2,\, x_3>b_1\}$ and $\Omega_2=\{\bm{x}\in\mathbb R^3:
0<x_1<L_1,\, 0<x_2<L_2,\, x_3<b_2\}$ the unbounded domains above and below
$\Omega$, respectively. Let $\Gamma_j=\{\bm{x}\in\mathbb R^3: 0<x_1<L_1,\,
0<x_2<L_2,\, x_3=b_j\}$ and $\Gamma_j'=\{\bm{x}\in\mathbb R^3: 0<x_1<L_1,\,
0<x_2<L_2,\, x_3=b_j'\}$,
where $b_j', j=1, 2$ are constants satisfying $b_2<b_2'<b_1'<b_1$. Let
$d_j=|b_j-b_j'|, j=1,2$. Define $\Omega'=\{\bm{x}\in\mathbb R^3:
0<x_1<L_1,\, 0<x_2<L_2,\, b_2'<x_3<b_1'\}$. Denote by
$\Omega_1'=\{\bm{x}\in\mathbb R^3:
0<x_1<L_1,\, 0<x_2<L_2,\, x_3\geq b_1'\}$ and $\Omega_2'=\{\bm{x}\in\mathbb
R^3: 0<x_1<L_1,\, 0<x_2<L_2,\, x_3\leq b_2'\}$ the unbounded domains above and
below $\Omega'$, respectively. The medium is assumed to be homogeneous away
from $\Omega'$, i.e., there exist constants $\varepsilon_j$ and $\mu_j$ such
that
\[
\varepsilon(\bm{x})=\varepsilon_j,\quad \mu(\bm{x})=\mu_j,
\quad\bm{x}\in\Omega_j',\quad j=1, 2.
\]
It is further assumed that $\varepsilon_1>0, \mu_j>0$, $j=1,2$ but
$\varepsilon_2$ may be complex for the substrate material in $\Omega_2'$.

\begin{figure}
\centering
\includegraphics[width=0.45\textwidth]{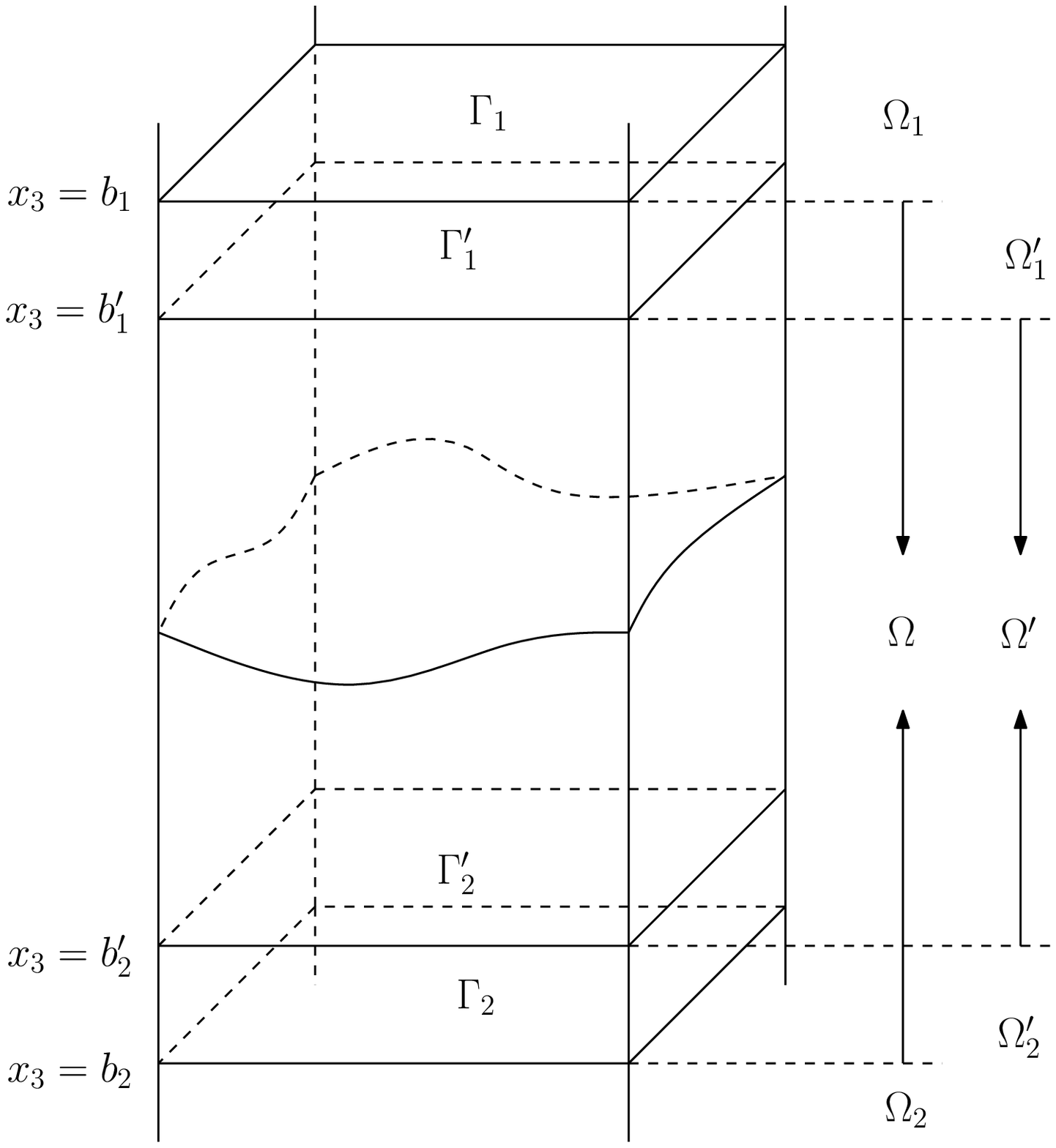}
\caption{Problem geometry of the diffraction grating in a biperiodic
structure.}
\label{pg}
\end{figure}

Let $(\bm{E}^{\rm inc}, \bm{H}^{\rm inc})$ be the incoming electromagnetic plane
waves that are incident on the grating surface from the top, where
\[
\bm{E}^{\rm inc}=\bm{p}e^{{\rm i}\bm{q}\cdot
\bm{x}},\quad \bm{H}^{\rm inc}=\bm{s}e^{{\rm i}\bm{q}\cdot \bm{x}},\quad
  \bm{s}=\frac{\bm{q}\times \bm{p}}{\omega\mu_1},\quad \bm{p}\cdot \bm{q}=0.
\]
Here
$\bm{q}=(\alpha_1, \alpha_2, -\beta)^\top=\omega\sqrt{\varepsilon_1\mu_1}
(\sin\theta_1\cos\theta_2, \sin\theta_1\sin\theta_2,
-\cos\theta_1)^\top$ and $\theta_1, \theta_2$ are incident angles satisfying
$0\leq\theta_1<\pi/2, 0\leq\theta_2<2\pi$.

Motivated by the uniqueness, we are interested in quasi-periodic solutions,
i.e., the phase shifted electromagnetic fields $(\bm{E}(\bm{x}),
\bm{H}(\bm{x}))e^{-{\rm i}(\alpha_1x_1+\alpha_2x_2)}$ are periodic functions in
$x_1$ and $x_2$ with periods $L_1$ and $L_2$, respectively.

Denote by $L^2(\Omega)$ the space of complex square integrable functions in
$\Omega$. Let
\[
H(\text{curl}, \Omega)=\{\bm{\varphi}\in L^2(\Omega)^3: \nabla\times
\bm{\varphi}\in L^2(\Omega)^3\}
\]
with the norm
\[
\|\bm{\varphi}\|_{H(\text{curl},
\Omega)}=\left(\|\bm{\varphi}\|_{L^2(\Omega)^3}^2
+\|\nabla\times\bm{\varphi}\|_{L^2(\Omega)^3}^2\right)^{1/2}.
\]
Define the periodic functional space
\eqn{
H_{\text{per}}(\text{curl}, \Omega)=\{\bm{\varphi}\in H(\text{curl},\Omega):\
&\bm{\varphi}(0,x_2,x_3)=\bm{\varphi}(L_1,x_2,x_3),\\
&\bm{\varphi}(x_1,0,
x_3)=\bm{\varphi}(x_1,L_2, x_3)\}.
}
Let
\[
H_{\text{qper}}(\text{curl}, \Omega)=\{\bm{\varphi}: \bm{\varphi}
e^{-{\rm i}(\alpha_1x_1+\alpha_2x_2)}\in H_{\text{per}}(\text{curl},
\Omega)\}.
\]
For any smooth vector field $\bm{\psi}=(\psi_1,\psi_2,\psi_3)^\top$,
denote by $\bm{\psi}_{\Gamma_j}=(\psi_1(x_1,x_2,b_j),
\psi_2(x_1,$ $x_2,b_j), 0)^\top$ its tangential component of $\bm\psi$
on the surface $\Gamma_j, j=1, 2$.

To describe the capacity operators and transparent boundary condition in the
formulation of the boundary value problem, we introduce some trace functional
spaces. Denote by $H^{-1/2}(\Gamma_j)$ the standard trace Sobolev space. Define
the quasi-biperiodic trace functional space as
\[
H^{-1/2}_{\text{qper}}(\Gamma_j)=\{\phi\in H^{-1/2}(\Gamma_j): \phi(\bm{x})
e^{-{\rm i}(\alpha_1x_1+\alpha_2x_2)}~\text{is biperiodic in $x_1$ and $x_2$}\}.
\]
Let $n=(n_1, n_2)^\top\in\mathbb{Z}^2$ and
\[
\alpha_{jn}=\alpha_j+2\pi n_j/L_j,\quad j=1,2.
\]
For any $\phi\in H^{-1/2}_{\text{qper}}(\Gamma_j)$, it has the following
Fourier series expansion
\[
\phi(x_1,x_2,b_j)=\sum_{n\in\mathbb{Z}^2}\phi_n(b_j)
e^{{\rm i}(\alpha_{1n}x_1+\alpha_{2n}x_2)}.
\]
The norm can be characterized by
\[
\|\phi\|_{H^{-1/2}_{\text{qper}}(\Gamma_j)}^2=L_1L_2\sum_{n\in\mathbb{Z}^2}
(1+|\alpha_n|^2)^{-1/2}|\phi_n(b_j)|^2.
\]
where $\alpha_n=(\alpha_{1 n}, \alpha_{2 n})^\top$.

For any vector field $\bm{\varphi}=(\varphi_1,\varphi_2,\varphi_3)^\top$
and scalar field $\psi$, denote by
\begin{align*}
\text{curl}_{\Gamma_j}\bm{\varphi}&=\partial_{x_1}\varphi_2-\partial_{x_2}
\varphi_1,\\
\nabla_{\Gamma_j}\psi&=(\partial_{x_1}\psi,\partial_{x_2}\psi,0)^\top,\\
\text{div}_{\Gamma_j}\bm{\varphi}&=\partial_{x_1}\varphi_1+\partial_{x_2}
\varphi_2,
\end{align*}
the surface scalar curl, the surface gradient, and the surface divergence on
$\Gamma_j, j=1, 2$, respectively. Introduce the following tangential
functional spaces:
\begin{align*}
  &TL^2(\Gamma_j)=\{\bm{\varphi}\in L^2(\Gamma_j)^3,\,\varphi_3=0\},\\
&TH_{\text{qper}}^{-1/2}(\text{curl},\Gamma_j)=\{\bm{\varphi}\in H^{-1/2}
(\Gamma_j)^3: \text{curl}_{\Gamma_j}\bm{\varphi}\in
H^{-1/2}(\Gamma_j),\,\varphi_3=0,\\
  &\hspace{3.8cm}\varphi_1, \varphi_2 \,\text{are biperiodic functions}\}.
\end{align*}
For any quasi-periodic tangential vector field $\bm{\varphi}$, it has the
Fourier series expansion
\begin{equation*}
\bm{\varphi}=\sum_{n\in\mathbb{Z}^2}(\varphi_{1n}, \varphi_{2n}, 0)^\top
e^{{\rm i}(\alpha_{1n}x_1+\alpha_{2n}x_2)}.
\end{equation*}
The $TL^2(\Gamma_j)$ norm of $\bm{\varphi}$ may be represented as
\begin{equation*}
  \|\bm{\varphi}\|^2_{TL^2(\Gamma_j)}=L_1L_2\sum\limits_{n\in Z^2}
  (|\varphi_{1n}|^2+|\varphi_{2n}|^2).
\end{equation*}
Using the Fourier coefficients, we may characterize the norm on the
space $TH_{\text{qper}}^{-1/2}(\text{curl},$ $\Gamma_j)$:
\begin{align}\label{normcurlgamma}
\|\bm{\varphi}\|^2_{TH_{\text{qper}}^{-1/2}(\text{curl},
\Gamma_j)}=L_1L_2\sum\limits_{n\in Z^2}
(1+|\alpha_{n}|^2)^{-1/2}\big[|\varphi_{1n}(b_j)|^2+|\varphi_{2n}
(b_j)|^2\notag\\
+|\alpha_{1n}\varphi_{2n}(b_j)-\alpha_{2n}\varphi_{1n}(b_j)|^2\big].
\end{align}

\subsection{Variational formulation}

In this section, we introduce the transparent boundary condition and
variational formulation for the diffraction grating problem. The details can be
found in \cite{BLW10} on the derivation of the TBC.

It follows from the radiation condition that the solution $(\bm{E}, \bm{H})$ of
the diffraction grating problem is composed of bounded outgoing plane waves in
$\Omega_1$ and $\Omega_2$, plus the incident wave $(\bm{E}^{\rm inc},
\bm{H}^{\rm inc})$ in $\Omega_1$. For convenience, we define
\[
\bm{E}^{\rm inc}_1=\bm{E}^{\rm inc},\quad \bm{H}^{\rm inc}_1=\bm{H}^{\rm
inc},\quad \bm{E}^{\rm inc}_2=\bm{H}^{\rm inc}_2=0.
\]
In virtue of the quasi-periodicity of $(\bm{E}, \bm{H})$, we get from the
Rayleigh expansion that
\[
\bm{E}-\bm{E}^{\rm inc}_j=\sum_{n\in\mathbb{Z}^2}\bm{p}_{jn}e^{{\rm
i}\bm{q}_{jn}\cdot\bm{x}},\quad
\bm{H}-\bm{H}^{\rm inc}_j=\sum_{n\in\mathbb{Z}^2}\bm{s}_{jn} e^{{\rm
i}\bm{q}_{jn}\cdot\bm{x}},\quad \bm{x}\in\Omega_j,
\]
where
\[
 \bm{s}_{jn}=\frac{1}{\omega\mu_j}\bm{q}_{jn}\times\bm{p}_{jn},\quad
\bm{p}_{jn}\cdot\bm{q}_{jn}=0,\quad \bm{q}_{jn}=(\alpha_{1n}, \alpha_{2n},
(-1)^{j-1}\beta_{jn})^\top.
\]
Here
\begin{equation}\label{beta}
 \beta_{jn}=(\kappa_j^2-|\alpha_n|^2)^{1/2},\quad {\rm Im}\beta_{jn}\geq
0,\quad \kappa_j^2=\omega^2\varepsilon_j\mu_j.
\end{equation}
We exclude possible resonances by assuming that
$\kappa_j^2\neq|\alpha_n|^2, n\in\mathbb{Z}^2, j=1,2.$

It follows from Rayleigh's expansions of $(\bm{E}, \bm{H})$ in $\Omega_j$ that
the transparent boundary conditions hold:
\begin{align*}
(\bm{H}-\bm{H}^{\rm inc})\times\bm{\nu}_1&=\mathscr{T}_1(\bm{E}-\bm{E}^{\rm
inc})_{\Gamma_1}, \quad \text{on}\ \Gamma_1,\\
\bm{H}\times\bm{\nu}_2&=\mathscr{T}_2\bm{E}_{\Gamma_2}, \quad \text{on}\
\Gamma_2,
\end{align*}
where $\bm{\nu}_j$ is the unit outward normal vector on $\Gamma_j$, i.e.,
$\bm{\nu}_j=(0,0,(-1)^{j-1})^\top$, and the capacity operator $\mathscr{T}_j$ is
defined as follows: for any tangential vector field $\bm{\varphi}\in
TH_{\text{qper}}^{-1/2}(\text{curl},\Gamma_j)$ which has Fourier  series
expansion
\[
\bm{\varphi}=\sum_{n\in\mathbb{Z}^2}(\varphi_{1n}^{(j)},\varphi_{2n}^{(j)},
0)^\top e^{{\rm i}(\alpha_{1n}x_1+\alpha_{2n}x_2)},
\]
we let
\begin{equation}\label{tbc}
\mathscr{T}_j\bm{\varphi}=\sum\limits_{n\in\mathbb{Z}^2}(r_{1n}^{(j)},
r_{2n}^{(j)} , 0)^\top e^{{\rm i}(\alpha_{1n}x_1+\alpha_{2n}x_2)},
\end{equation}
where
\begin{align*}
r_{1n}^{(j)}&=\frac{1}{\omega\mu_j\beta_{jn}}\big[(\kappa_j^2-\alpha_{2n}
^2)\varphi_{1n}^{(j)}+\alpha_{1n}\alpha_{2n}\varphi_{2n}^{(j)}\big],\\
r_{2n}^{(j)}&=\frac{1}{\omega\mu_j\beta_{jn}}\big[(\kappa_j^2-\alpha_{1n}^2)
\varphi_{2n}^{(j)}+\alpha_{1n}\alpha_{2n}\varphi_{1n}^{(j)}\big].
\end{align*}

Now we present a variational formulation of the Maxwell system \eqref{me} in
the space $H_{\text{qper}}(\text{curl},\Omega)$. Eliminating the magnetic
field $\bm{H}$ from \eqref{me}, we obtain
\begin{equation}\label{mee}
\nabla\times({\mu}^{-1}\nabla\times\bm{E})-\omega^2\varepsilon\bm{E}=0\quad
\text{in}\ \Omega.
\end{equation}
Multiplying the complex conjugate of a test function $\bm\psi$ in
$H_{\text{qper}}(\text{curl},\Omega)$, integrating over $\Omega$, and using
integration by parts, we arrive at the variational form for the scattering
problem: Find $\bm{E}\in H_{\text{qper}}(\text{curl},\Omega)$ such that
\begin{equation}\label{vp}
  a(\bm{E}, \bm{\psi})=\langle
\bm{f}, \bm{\psi}\rangle_{\Gamma_1},\quad\forall\bm{\psi}\in
  H_{\text{qper}}(\text{curl},\Omega),
\end{equation}
where the sesquilinear form
\begin{equation}\label{sf}
a(\bm{\varphi}, \bm{\psi})=\int_{\Omega}\mu^{-1}\nabla\times\bm{\varphi}
\cdot\nabla\times\bar{\bm \psi}
  -\omega^{2}\int_{\Omega}\varepsilon\varphi\cdot\bar{\bm\psi}
  -{\rm i}\omega\sum_{j=1}^{2}\int_{\Gamma_j}
  \mathscr{T}_j{\bm\varphi}_{\Gamma_j}\cdot\bar{\bm\psi}_{\Gamma_j},
\end{equation}
and the linear functional
\[
\langle\bm{f},\bm{\psi}\rangle_{\Gamma_1}={\rm i}\omega\int_{\Gamma_1}
(\bm{H}^{\rm inc}\times\bm{\nu}_{1}-\mathscr{T}_1 \bm{E}^{\rm
inc}_{\Gamma_1})\cdot\bar{\bm\psi}_{\Gamma_1}=-2{\rm
i}\omega\int_{\Gamma_1}\mathscr{T}_1 \bm{E}^{\rm
inc}_{\Gamma_1}\cdot\bar{\bm\psi}_{\Gamma_1}.
\]
Here we have used the identity
\[
\bm{H}^{\rm inc}\times{\bm\nu}_{1}=-\mathscr{T}_1 \textbf{E}^{\rm
inc}_{\Gamma_1} \quad\text{on}\ \Gamma_1.
\]

We assume that the variational problem \eqref{vp} admits a unique weak
solution in $H_{\text{qper}}(\text{curl},\Omega)$. Then it follows from the
general theory in Babu\v{s}ka and Aziz \cite{BA73} that there
exists a constant $\gamma_1>0$ such that the following inf-sup condition holds:
\begin{equation}\label{isc}
\sup_{0\neq\bm{\psi}\in
H_{\text{qper}}(\text{curl},\Omega)}\frac{|a(\bm{\varphi}, \bm{\psi})|}
{\|\bm{\psi}\|_{H(\text{curl},\Omega)}}
\geq\gamma_1\|\bm{\varphi}\|_{H(\text{curl} , \Omega)},\quad\forall\bm{\varphi}\in
H_{\text{qper}}(\text{curl},\Omega).
\end{equation}

\section{The a posteriori estimate}

In this section, we introduce the finite element approximation and present the a
posteriori error estimate which plays an important role for the adaptive finite
element method.

Let $\mathcal{M}_h$ be a regular tetrahedral mesh of the domain $\Omega$. To
deal with the quasi-periodic boundary conditions, we assume further that the
mesh is periodic in both $x_1$ and $x_2$ directions, i.e., the projection of the
surface mesh on any face of $\Omega$ perpendicular to the $x_1$-axis or
the $x_2$-axis into its opposite face coincides with the surface mesh on the
opposite face.

Denote by $\mathcal{F}_h$ the set of all faces of tetrahedrons in
$\mathcal{M}_h$. Let
$V_h\subset H_{\text{qper}}(\text{curl},\Omega)$ be an edge element
space that contains
the lowest order N\'{e}d\'{e}lec edge element space
\begin{equation}\label{ees}
  V_h=\{\bm{v}_h\in H_{\text{qper}}(\text{curl},\Omega):
\bm{v}_h|_T=\bm{a}_T+\bm{b}_T\times \bm{x},\, \bm{a}_T,
\bm{b}_T\in\mathbb{C}^3,\,\forall T\in\mathcal{M}_h\}.
\end{equation}
The finite element approximation to the problem \eqref{vp} reads as
follows: Find $\bm{E}_h\in V_h$ such that
\begin{equation}\label{fep}
  a(\bm{E}_h, \bm{\psi}_h)=\langle \bm{f}, \bm{\psi}_h\rangle_{\Gamma_1},\quad
\forall\bm{\psi}_h\in V_h.
\end{equation}

In the above formulation, the capacity operators $\mathscr{T}_j$ given by
\eqref{tbc} is defined by an infinite series which is unrealistic in actual
calculations. It is necessary to truncate the nonlocal operator by taking
sufficiently many terms of the expansions so as to attain our feasible
algorithm. We truncate the capacity operator $\mathscr{T}_j$ as follows:
\begin{equation}\label{ttbc}
  \mathscr{T}_j^{N_j}\bm{\varphi}=\sum_{n\in U_{N_j}}
(r_{1n}^{(j)},r_{2n}^{(j)},0)^\top e^{{\rm i}(\alpha_{1n}x_1+\alpha_{2n}x_2)},
\end{equation}
where the index set $U_{N_j}$ is defined as
\begin{equation}\label{Uj}
  U_{N_j}=\bigg\{n=(n_1,n_2)^\top\in\mathbb{Z}^2:\
|\al_n|\leq \frac{2\pi}{\LL}N_j\bigg\},\quad j=1,2.
\end{equation}
Roughly speaking, the points in $U_{N_j}$ occupy an area of $\pi N_j^2$.

Now we are ready to define the truncated finite element formulation which
leads to the discrete approximation to \eqref{vp}: Find
$\bm{E}_h^N\in V_h$ such that
\begin{equation}\label{fem}
a_N(\bm{E}_h^N, \bm{\psi}_h)=\langle\bm{f}^{N_f},
\bm{\psi}_h\rangle_{\Gamma_1},\quad\forall \bm{\psi}_h\in V_h,
\end{equation}
where the sesquilinear form $a_N$: $V_h\times V_h\rightarrow\mathbb{C}$ is
defined as follows:
\begin{equation}\label{tsf}
a_N(\bm{\varphi}, \bm{\psi})=\int_{\Omega}\mu^{-1}
\nabla\times\bm{\varphi}\cdot\nabla\times\bar{\bm\psi}
  -\omega^2\int_{\Omega}\varepsilon\bm{\varphi}\cdot\bar{\bm\psi}
  -{\rm i}\omega\sum_{j=1}^{2}\int_{\Gamma_j}
  \mathscr{T}_j^{N_j}\bm{\varphi}_{\Gamma_j}\cdot\bar{\bm\psi}_{\Gamma_j},
\end{equation}
and
\begin{equation*}
  \langle f^{N_f},\bm{\psi}_h\rangle_{\Gamma_1}=-2\bm{i}\omega\int_{\Gamma_1}
  \mathscr{T}_1^{N_f}\bm{E}_{\Gamma_1}^{\rm inc}\cdot(\bar{\psi}_h)_{\Gamma_1}.
\end{equation*}

For any $T\in\mathcal{M}_h$, we define the residuals
\begin{align*}
  & R_T^{(1)}:=\omega^2\varepsilon\bm{E}_h^N|_T
  -\nabla\times(\mu^{-1}\nabla\times\bm{E}_h^N|_T),\\
  &R_T^{(2)}:=-\omega^2\nabla\cdot(\varepsilon\bm{E}_h^N|_T).
\end{align*}
Given an interior face $F\in\mathcal{F}_h$, which is the common face of
$T_1$ and $T_2$, we define the jump residuals across $F$ as
\begin{align*}
  & J_F^{(1)}:=\mu^{-1}(\nabla\times\bm{E}_h^N|_{T_1}
  -\nabla\times\bm{E}_h^N|_{T_2})\times\bm{\nu}_F,\\
  &J_F^{(2)}:=\omega^2(\varepsilon\textbf{E}_h^N|_{T_2}
  -\varepsilon\bm{E}_h^N|_{T_1})\cdot\bm{\nu}_F,
\end{align*}
where the unit normal vector $\bm{\nu}_F$ on $F$ points from $T_2$ to $T_1$.
Given a face $F\in\mathcal{F}_h\cap\Gamma_1$, we define the residuals as
\begin{align*}
  &J_F^{(1)}=2[-(\mu^{-1}\nabla\times\bm{E}_h^N)\times\bm{\nu}_1
  +{\rm i}\omega \mathscr{T}_1^{N_1}(\bm{E}_h^N)_{\Gamma_1}
  -2{\rm i}\omega\mathscr{T}_1 \bm{E}^{\rm inc}_{\Gamma_1}]\\
  &J_F^{(2)}=2[\omega^2\varepsilon\bm{E}_h^N\cdot\bm{\nu}_1
-{\rm
i}\omega\text{div}_{\Gamma_1}(\mathscr{T}_1^{N_1}(\textbf{E}_h^N)_{\Gamma_1 })
  +2{\rm
i}\omega\text{div}_{\Gamma_1}(\mathscr{T}_1 \bm{E}^{\rm inc}_{\Gamma_1})].
\end{align*}
Given a face $F\in\mathcal{F}_h\cap\Gamma_2$, we define the residuals as
\begin{align*}
&J_F^{(1)}=2[-(\mu^{-1}\nabla\times\bm{E}_h^N)\times\bm{\nu}_2+{\rm i}\omega
\mathscr{T}_2^{N_2}(\bm{E}_h^N)_{\Gamma_2}]\\
&J_F^{(2)}=2[\omega^2\varepsilon\bm{E}_h^N\cdot\bm{\nu}_2-{\rm
i}\omega\text{div} _{\Gamma_2}(\mathscr{T}_2^{N_2}(\bm{E}_h^N)_{\Gamma_2
})].
\end{align*}

Define
\begin{align*}
  &\Gamma_{10}=\{\bm{x}\in\mathbb R^3: x_1=0,\,0<x_2<L_2,\,b_2<x_3<b_1\},\\
  &\Gamma_{11}=\{\bm{x}\in\mathbb R^3: x_1=L_1,\,0<x_2<L_2,\,b_2<x_3<b_1\},\\
  &\Gamma_{20}=\{\bm{x}\in\mathbb R^3: x_2=0,\,0<x_1<L_1,\,b_2<x_3<b_1\},\\
  &\Gamma_{21}=\{\bm{x}\in\mathbb R^3: x_2=L_2,\,0<x_1<L_1,\,b_2<x_3<b_1\}.
\end{align*}
For any face $F\in\mathcal{F}_h\cap \Gamma_{l0}$, let $F'\in\mathcal{F}_h$ be
the corresponding face on $\Gamma_{l1}\; (l=1,2)$, and let $T,T'\in
\mathcal{M}_h$ be the two elements such that $T\supset F$ and $T'\supset F'$. 
We define the jump residuals across $F$ and $F'$ as 
\begin{align*}
& J_F^{(1)}:=\mu^{-1}\big(e^{-{\rm i}\alpha_jL_j}\nabla\times\bm{E}_h^N|_{T'}
  -\nabla\times\bm{E}_h^N|_{T}\big)\times\bm{\nu}_F,\\
& J_{F'}^{(1)}:=\mu^{-1}\big(\nabla\times\bm{E}_h^N|_{T'}
  -e^{{\rm i}\alpha_jL_j}\nabla\times\bm{E}_h^N|_{T}\big)\times\bm{\nu}_F,\\
  &J_F^{(2)}:=\omega^2(\varepsilon\textbf{E}_h^N|_{T}
  -e^{-{\rm i}\alpha_jL_j}\varepsilon\bm{E}_h^N|_{T'})\cdot\bm{\nu}_F,\\
    &J_{F'}^{(2)}:=\omega^2(e^{{\rm i}\alpha_jL_j}\varepsilon\textbf{E}_h^N|_{T}
  -\varepsilon\bm{E}_h^N|_{T'})\cdot\bm{\nu}_F,
\end{align*}
where $\bm{\nu}_F$ is the  unit outward normal vector  to $F$.

For any $T\in\mathcal{M}_h$, denote by $\eta_T$ the local error estimator, which
is defined as follows:
\[
\eta_T^2=h_T^2\left(\|R_T^{(1)}\|_{L^2(T)^3}^2+\|R_T^{(2)}\|_{L^2(T)}
^2\right)+h_T\sum_{F\subset\partial T}
 \left (\|J_F^{(1)}\|_{L^2(F)^3}^2+\|J_F^{(2)}\|_{L^2(F)}^2\right).
\]

We now state the main result of this paper.

\begin{theorem}\label{mt}
Let $\bm{E}$ and $\bm{E}_h^N$ be the solutions of \eqref{vp} and
\eqref{fem}, respectively. Then there exist two integers $M_j$, $j=1,2$
independent of $h$ and satisfying $\big(\frac{2\pi M_j}{\LL}\big)^2>{\rm
Re}\kappa_j^2$ such that for $N_j\geq M_j$ the following a posteriori error
estimate holds:
\begin{align*}
\|\bm{E}-\bm{E}_h^N\|_{H({\rm curl},\Omega)}\leq
C\Bigg(\bigg(\sum_{T\in\mathcal{M}_h}\eta_T^2\bigg)^{1/2}
    +\sum\limits_{j=1}^2e^{-d_j\si_j}
    \|\bm{E}^{\rm inc}\|_{TL^2(\Gamma_1)}\Bigg),
\end{align*}
where $\si_j=\big(\big(\frac{2\pi N_j}{\LL}\big)^2-{\rm
Re}\kappa_j^2\big)^{1/2}$ and the constant $C$ is independent of $h$ and $M_j$.
\end{theorem}

\section{Proof of the main theorem} The section is devoted to the proof of
Theorem \ref{mt}.

\subsection{The dual problem}

Denote the error by $\bm{\xi}=\bm{E}-\bm{E}_h^N$. Consider the Helmholtz decomposition
\begin{align}\label{zeta}
\ep\bm{\xi}=\ep\nabla q+\bm{\zeta},\quad q\in H_0^1(\Omega), \quad \text{div}
\bm{\zeta}=0.
\end{align}
Introduce the following dual
problem to the original scattering problem: Find $\bm{W}\in
H_{\text{qper}}(\text{curl},\Omega)$ such that it satisfies the variational
problem
\begin{equation}\label{dp}
a(\bm{v}, \bm{W})=(\bm{v}, \bm{\ze}),\quad \forall
\bm{v}\in H_{\text{qper}}(\text{curl},\Omega).
\end{equation}
It is easy to verify that $\bm{W}$ is the weak solution of the
boundary value problem:
\[
\begin{cases}
\nabla\times(({\bar\mu})^{-1}\nabla\times\bm{W})-\omega^2\bar\varepsilon
\bm{W}=\bm{\ze }\quad &\text{in}\ \Omega,\\
(({\bar\mu})^{-1}\nabla\times\bm{W})\times \bm{\nu}_{j}=-{\rm i}\omega
\mathscr{T}_j^{*}\bm{W}_{\Gamma_j}\quad &\text{on}\ \Gamma_j,
  \end{cases}
\]
where the adjoint operator $\mathscr{T}_j^*$ takes the following form:
\begin{equation}\label{atbc}
\mathscr{T}_j^*\bm{\varphi}=\sum_{n\in\mathbb{Z}^2}(\rho_{1n}^{(j)},\rho_{2n}^{
(j)}, 0)^\top e^{{\rm i}(\alpha_{1n}x_1+\alpha_{2n}x_2)},\quad j=1,2.
\end{equation}
Here
\begin{align*}
\rho_{1n}^{(j)}&=\frac{1}{\omega\bar{\mu}_j\bar{\beta}_{jn}}
\big[(\bar{\kappa}_j^2-\alpha_{2n}^2)\varphi_{1n}^{(j)}+\alpha_{1n}\alpha_{2n}
\varphi_{2n}^{(j)}\big],\\
\rho_{2n}^{(j)}&=\frac{1}{\omega\bar{\mu}_j\bar{\beta}_{jn}}
 \big[(\bar{\kappa}_j^2-\alpha_{1n}^2)\varphi_{2n}^{(j)}+\alpha_{1n}
\alpha_ {2n}\varphi_{1n}^{(j)}\big].
\end{align*}
Assuming that the dual problem has a unique weak solution, we have the
stability estimate
\begin{equation}\label{dps}
 \|\bm{W}\|_{H(\text{curl},\Omega)}\leq C_0\|\bm{\ze}\|_{L^2(\Omega)^3},
\end{equation}
where $C_0$ is a positive constant.

Denote by $U_h\subset H^{1}(\Omega)$ the standard continuous piecewise
linear finite element space. Clearly, we have
\[
\nabla U_h\subseteq V_h,
\]
where $V_h$ is the lowest order N\'{e}d\'{e}lec edge element space
defined in (\ref{ees}).

\subsection{Error representation formula}

The following lemma shows that $\|\bm{\xi}\|_{H({\rm {curl}},\Omega)}$
can be bounded by $\|\bm{\xi}\|_{L^2(\Omega)^3}$ and vice versa.
\begin{lemma}
Let $\bm{E}$, $\bm{E}_h^N$, and $\bm{W}$ be the solutions to the problems
\eqref{vp}, \eqref{fem}, and \eqref{dp}, respectively. Then we have
\begin{align}\label{xi-curl}
\|\bm{\xi}\|_{H({\rm {curl}},\Omega)}^2\lesssim&{\rm{Re}}\bigg[a(\bm\xi,\bm\xi)+{\rm
i}\omega\sum_{j=1}^2\int_{\Gamma_j}
(\mathscr{T}_j-\mathscr{T}_j^{N_j})\bm{\xi}_{\Gamma_j}\cdot\bar{\bm\xi}_{
\Gamma_j}\bigg] \\
&-\omega\sum_{j=1}^2{\rm
Im}\int_{\Gamma_j}\mathscr{T}_j^{N_j}\bm{\xi}_{\Gamma_j}
\cdot\bar{\bm\xi}_{\Gamma_j }
 +{\rm Re}\int_\Omega (\omega^2\varepsilon+{\bf
I})\bm{\xi}\cdot\bar{\bm \xi}\nonumber,
\end{align}
\begin{align}\label{xi-L2}
(\ep\bm{\xi},\bm{\xi})=&\bigg[a(\bm{\xi}, \bm{W})+{\rm
i}\omega\sum_{j=1}^2\int_{\Gamma_j}(\mathscr{T}_j-\mathscr{T}_j^{N_j})\bm{\xi}_{
\Gamma_j}\cdot\bar{\bm W}_{\Gamma_j}\bigg]\\
&+(\ep\bm{\xi},\na q)-{\rm
i}\omega\sum_{j=1}^2\int_{\Gamma_j}(\mathscr{T}_j-\mathscr{T}_j^{N_j})
\bm{\xi}_{\Gamma_j}\cdot\bar{\bm W}_{\Gamma_j},\notag
\end{align}
and for any $\bm{\psi}_h\in V_h, \bm{\psi}\in
H_{\rm qper}({\rm curl},\Omega), q_h\in U_h\cap H_0^1(\Om)$, there hold
\begin{align}
a(\bm{\xi}, \bm{\psi})&+{\rm
i}\omega\sum_{j=1}^2\int_{\Gamma_j}(\mathscr{T}_j-\mathscr{T}_j^{N_j})\bm{\xi}_{
\Gamma_j}\cdot\bar{\bm\psi}_{\Gamma_j}\nonumber\\
&=\langle\bm{f},
\bm{\psi}-\bm{\psi}_h\rangle_{\Gamma_1}-a_N(\bm{E}_h^N,
\bm{\psi}-\bm{\psi}_h)+{\rm i}\omega\sum_{j=1}^2\int_{\Gamma_j}(\mathscr{T}_j-\mathscr{T}_j^{N_j} )
\bm{E}_{\Gamma_j}\cdot\bar{\bm\psi}_{\Gamma_j},
\label{aTT}\\
&\qquad(\ep\bm{\xi},\na q)=-(\ep\bm{E}_h^N,\na (q-q_h)).\label{xiq}
\end{align}
\end{lemma}

\begin{proof}
The inequality \eqref{xi-curl} follows from the definition of the
sesquilinear form $a$ in \eqref{sf}. The identity \eqref{xi-L2} follows by
taking $\bm{v}=\bm{\xi}$ in \eqref{dp} and using \eqref{zeta}. It remains to prove (\ref{aTT}) and \eqref{xiq}. Using
\eqref{vp} and \eqref{fem}, we obtain
\begin{align*}
a(\bm{\xi},\bm{\psi})=&a(\bm{E}-\bm{E}_h^N, \bm{\psi}-\bm{\psi}_h)+a(\bm{E}
-\bm{E}_h^N ,\bm{\psi}_h)\\
=&a(\bm{E}, \bm{\psi}-\bm{\psi}_h)-a(\bm{E}_h^N,
\bm{\psi}-\bm{\psi}_h)+a(\bm{E}-\bm{E}_h^N, \bm{\psi}_h)\\
    =&\langle\bm{f},
\bm{\psi}-\bm{\psi}_h\rangle_{\Gamma_1}-a_N(\bm{E}_h^N, \bm{\psi}-\bm{\psi}_h)\\
    &\ +a_N(\bm{E}_h^N, \bm{\psi}-\bm{\psi}_h)-a(\bm{E}_h^N,
\bm{\psi}-\bm{\psi}_h)+a(\bm{E}-\bm{E}_h^N, \bm{\psi}_h)\\
     =&\langle\bm{f},
\bm{\psi}-\bm{\psi}_h\rangle_{\Gamma_1}-a_N(\bm{E}_h^N, \bm{\psi}-\bm{\psi}_h)\\
    &\ +a_N(\bm{E}_h^N, \bm{\psi}-\bm{\psi}_h)-a(\bm{E}_h^N,
\bm{\psi}-\bm{\psi}_h)+\langle\bm{f},
\bm{\psi}_h\rangle_{\Gamma_1}-a(\bm{E}_h^N, \bm{\psi}_h)\\
 =&\langle\bm{f}, \bm{\psi}-\bm{\psi}_h\rangle_{\Gamma_1}
    -a_N(\bm{E}_h^N, \bm{\psi}-\bm{\psi}_h)\\
    &\ +a_N(\bm{E}_h^N, \bm{\psi})-a(\bm{E}_h^N, \bm{\psi})\\
  =&\langle\bm{f},\bm{\psi}-\bm{\psi}_h\rangle_{\Gamma_1}-a_N(\bm{E}_h^N,
\bm{\psi}-\bm{\psi}_h)\\
    &\ -\sum_{j=1}^2{\rm i}\omega\int_{\Gamma_j}
      (\mathscr{T}_j-\mathscr{T}_j^{N_j})\bm{\xi}_{\Gamma_j}\cdot\bar{\bm\psi}_{\Gamma_j}
+{\rm i}\omega\sum_{j=1}^2\int_{\Gamma_j}(\mathscr{T}_j-\mathscr{T}_j^{N_j}
)\bm{E}_{\Gamma_j} \cdot\bar{\bm\psi}_{\Gamma_j},
 \end{align*}
which implies \eqref{aTT}. By taking $\bm{\psi}=\bm{\psi}_h=\na q_h, \forall
q_h\in U_h\cap H_0^1(\Om)$ in the above identity and using \eqref{sf}, we
conclude that
\begin{align*}
(\ep\xi,\na q_h)=0.
\end{align*}
Then \eqref{xiq} follows by noting that ${\rm div} (\bm{\ep}\bm{E})=0$.
This completes the proof of the lemma.
\end{proof}

\subsection{Several trace results}

The following lemma is a trace regularity result for
$H_{\text{qper}}(\text{curl},\Omega)$. The proof can be found in
\cite{BLW10}.
\begin{lemma}\label{tt}
Let $\gamma_0=\max\{\sqrt{1+(b_1-b_2)^{-1}},\sqrt{2}\}$. Then the following
estimate holds:
\[
\|\bm{\psi}_{\Gamma_j}\|_{TH_{\rm{qper}}^{-1/2}({\rm curl},\Gamma_j)}
\leq\gamma_0\|\bm{\psi}\|_{H({\rm curl},\Omega)},\quad\forall
\bm{\psi}\in H_{\rm{qper}}({{\rm curl}},\Omega).
\]
\end{lemma}

\begin{lemma}\label{tr}
For any $\de>0$, there exits a constant $C$ depending only on $\de$, $b_1$,
and $b_2$ such that the following estimate holds:
\[
\|\bm{\psi}_{\Gamma_j}\|_{H^{-1/2}_{\rm
qper}(\Gamma_j)^3}^2\leq\de\|\nabla\times\bm{ \psi }
\|_{L^2(\Omega)^3}^2+C(\de)\|\bm{\psi}\|_{L^2(\Omega)^3}^2,\quad\forall\,
\bm{\psi}\in H_{\rm qper}({\rm curl},\Omega).
\]
\end{lemma}

\begin{proof}
First we have
\begin{align*}
    &(b_1-b_2)|\zeta(b_j)|^2\\
=&\int_{b_2}^{b_1}|\zeta(x_3)|^2{\rm d}x_3+\int_{b_2}^{b_1}
    \int_{x_3}^{b_j}\frac{{\rm d}}{{\rm d}\tau}|\zeta(\tau)|^2{\rm d}\tau{\rm
d}x_3\\
\leq&\int_{b_2}^{b_1}|\zeta(x_3)|^2{\rm d}x_3+(b_1-b_2)\int_{b_2}^{b_1}
    2|\zeta(x_3)||\zeta^{'}(x_3)|{\rm d}x_3\\
\leq&\int_{b_2}^{b_1}|\zeta(x_3)|^2{\rm d}x_3+(b_1-b_2)
\bigg(\frac{2(1+|\alpha_n|^2)^{1/2}}{\de}\int_{b_2}^{b_1}|\zeta(x_3)|^2{\rm
d}x_3\\
    &\hspace{2cm}+\frac{\de}{2(1+|\alpha_n|^2)^{1/2}}
    \int_{b_2}^{b_1}|\zeta^{'}(x_3)|^2{\rm d}x_3\bigg),
\end{align*}
which implies that
\begin{align}\label{Ceta}
 &(1+|\alpha_n|^2)^{-1/2}|\zeta(b_j)|^2\notag\\
\leq&\ C(\de)\int_{b_2}^{b_1}|\zeta(x_3)|^2{\rm
d}x_3+\frac{\de}{2}(1+|\alpha_n|^2)^{-1}\int_{b_2}^{b_1}|\zeta^{'}(x_3)|^2{\rm
d}x_3.
  \end{align}

Given $\bm{\psi}\in H_{{\rm qper}}({\rm curl},\Omega)$, it has the expansion
\begin{equation*}
\bm{\psi}(\bm{x})=\sum_{n\in\mathbb{Z}^2}(\psi_{1n}(x_3),\psi_{2n}(x_3),
\psi_{3n} (x_3))e^{{\rm i}(\alpha_{1n}x_1+\alpha_{2n}x_2)},
  \end{equation*}
  which yields that
  \[
    \|\bm{\psi}_{\Gamma_j}\|_{H_{\rm qper}^{-1/2}(\Gamma_j)}^2
    =L_1L_2\sum\limits_{n\in Z^2}(1+|\alpha_n|^2)^{-1/2}
    (|\psi_{1n}(b_j)|^2+|\psi_{2n}(b_j)|^2).
  \]
Using \eqref{Ceta} gives
  \[
    (1+|\alpha_n|^2)^{-1/2}(|\psi_{1n}(b_j)|^2+|\psi_{2n}(b_j)|^2)\leq I_1+I_2,
  \]
  where
  $$I_1=C(\de)\int_{b_2}^{b_1}(|\psi_{1n}(x_3)|^2+|\psi_{2n}(x_3)|^2){\rm
d}x_3$$
  and
  \begin{align*}
    I_2=&\frac{\de}{2}(1+|\alpha_n|^2)^{-1}\int_{b_2}^{b_1}
    (|\psi_{1n}'(x_3)|^2+|\psi_{2n}'(x_3)|^2){\rm d}x_3\\
    \leq&\de(1+|\alpha_n|^2)^{-1}\int_{b_2}^{b_1}\big[
    |\psi_{1n}'(x_3)-{\rm i}\alpha_{1n}\psi_{3n}(x_3)|^2
    +|\psi_{2n}'(x_3)-{\rm i}\alpha_{2n}\psi_{3n}(x_3)|^2\\
    &+(\alpha_{1n}^2+\alpha_{2n}^2)|\psi_{3n}(x_3)|^2\big]{\rm d}x_3\\
\leq&\de\int_{b_2}^{b_1}\big[|\psi_{1n}'(x_3)-{\rm i}\alpha_{1n}\psi_{3n}
(x_3)|^2+|\psi_{2n}'(x_3)-{\rm i}\alpha_{2n}\psi_{3n}(x_3)|^2
    +|\psi_{3n}(x_3)|^2\big]{\rm d}x_3,
  \end{align*}
  which completes the proof after summing over $n\in\mathbb{Z}^2$.
\end{proof}

The following lemma gives an estimate for quasi-periodic
divergence-free functions.

\begin{lemma}\label{L-div}
Let $\bm{v}=(v_{1n}(x_3),v_{2n}(x_3),v_{3n}(x_3))^\top
e^{{\rm i}(\alpha_{1n}x_1+\alpha_{2n}x_2)}$. Suppose ${\rm div}\, \bm{v}=0$.
Then for any $\de>0$, the following estimate holds:
\[
\bigg|\int_{\Ga_j}v_{3n}'\bar v_{3n}\bigg|\le |\al_n|^2(1+\de)
\norm{\bm{v}}_{L^{2}(\tilde\Om_j)^3}^2+C\de^{-1} d_j^{-2}
\|\bm{v}\|_{H({\rm curl},\tilde\Omega_j)}^2,\quad j=1,2,
\]
where $\tilde\Om_j=\Om_j'\setminus\Om_j$ and $d_j=|b_j-b_j'|.$
\end{lemma}

\begin{proof} We only prove the case of $j=1$ since the proof for $j=2$ is similar. Clearly, $v_{3n}'=-{\rm i}\al_{1n}v_{1n}-{\rm i}\al_{2n}v_{2n}$ and
\eqn{\norm{\na\times\bm{v}}_{L^2(\tilde\Om_1)^3}=L_1L_2\Big(&\norm{v_{1n}'-{\rm i}\al_{1n}v_{3n}}_{L^{2}(B_1)}^2+\norm{v_{2n}'-{\rm i}\al_{2n}v_{3n}}_{L^{2}(B_1)}^2\\
&+\norm{\al_{1n}v_{2n}-\al_{2n}v_{1n}}_{L^{2}(B_1)}^2\Big),}
where $B_1=[b_1',b_1]$. Let $\chi\in C^2(B_1)$ be a function satisfying
\eq{\label{chi}
\chi(x_3)\begin{cases}
=1, & x_3\in [b_1'+\frac23 d_1,b_1],\\
=0, &x_3\in [b_1',b_1'+\frac13 d_1],\\
\in [0,1], &x_3\in (b_1'+\frac13 d_1, b_1'+\frac23 d_1),
\end{cases}
 \quad |\chi'|\ls d_1^{-1},\text{ and }  |\chi''|\ls d_1^{-2}.}
Let $\bm{w}=\chi\bm{v}$.   We conclude that
\eqn{&\abs{v_{3n}'(b_1)\bar v_{3n}(b_1)}=\abs{w_{3n}'(b_1)\bar w_{3n}(b_1)}=\bigg|\int_{b_1'}^{b_1}\big(w_{3n}''\bar w_{3n}+|w_{3n}'|^2\big)\bigg|\\
=&\bigg|\int_{b_1'}^{b_1}\big((\chi v_{3n}''+2\chi'v_{3n}'+\chi'' v_{3n})\chi\bar v_{3n}+|\chi v_{3n}'+\chi' v_{3n}|^2\big)\bigg|\\
=&\bigg|\int_{b_1'}^{b_1}\big(\chi^2 v_{3n}''\bar v_{3n}+\chi^2|v_{3n}'|^2+\chi'\chi(3v_{3n}'\bar v_{3n}+v_{3n}\bar v_{3n}')+(\chi''\chi+|\chi'|^2)|v_{3n}|^2\big)\bigg|\db\\
=&\bigg|\int_{b_1'}^{b_1}\Big(\chi^2\big({\rm i}\al_{1n}({\rm i}\al_{1n}v_{3n}-v_{1n}')+{\rm i}\al_{2n}({\rm i}\al_{2n}v_{3n}-v_{2n}')\big)\bar v_{3n}+\chi^2|\al_n|^2|v_{3n}|^2\\
&+\chi^2|v_{3n}'|^2+\chi'\chi(3v_{3n}'\bar v_{3n}+v_{3n}\bar v_{3n}')+(\chi''\chi+|\chi'|^2)|v_{3n}|^2\Big)\bigg|\db\\
\le&(1+\de)\big(|\al_n|^2\norm{v_{3n}}_{L^2(B_1)}^2+\norm{v_{3n}'}_{L^2(B_1)}^2\big)\\
&+C\de^{-1}\big(\norm{{\rm i}\al_{1n}v_{3n}-v_{1n}'}_{L^2(B_1)}^2+\norm{{\rm i}\al_{2n}v_{3n}-v_{2n}'}_{L^2(B_1)}^2+d_1^{-2}\norm{v_{3n}}_{L^2(B_1)}^2\big)}
which together with $\norm{v_{3n}'}_{L^2(B_1)}^2\le
|\al_n|^2\big(\norm{v_{1n}}_{L^2(B_1)}^2+\norm{v_{2n}}_{L^2(B_1)}^2\big)$
implies that
\begin{align*}\bigg|\int_{\Ga_1}v_{3n}'\bar
v_{3n}\bigg|&=L_1L_2\abs{v_{3n}'(b_1)\bar v_{3n}(b_1)}\\
&\le |\al_n|^2(1+\de)\norm{\bm{v}}_{L^{2}(\tilde\Om_1)^3}^2+C\de^{-1} d_1^{-2}
\|\bm{v}\|_{H(\text{curl},\tilde\Omega_1)}^2,
\end{align*}
which completes the proof.
\end{proof}

\subsection{Estimates of \eqref{aTT} and \eqref{xiq}}

We first discuss the exponentially decay property of the evanescent modes.

Let $\bm{E}=(E_1, E_2, E_3)^\top$ be the solution to the variational problem
\eqref{vp}. Since $E_l(\bm{x})$ is a quasi-periodic function
with respect to $x_1$ and $x_2$, it has the following Fourier expansion:
\begin{equation}\label{Ej}
E_l(\bm{x})=\sum_{n\in\mathbb{Z}^2} E_{ln}(x_3)
e^{{\rm i}(\alpha_{1n}x_1+\alpha_{2n}x_2)},\quad l=1,2,3.
\end{equation}

The following lemma is crucial to derive the truncation error.

\begin{lemma}\label{expo deca}
Let $\bm{E}=(E_1, E_2, E_3)^\top$ be the solution to \eqref{vp} and
$E_{ln}$ is the Fourier coefficient given in \eqref{Ej}. If ${\rm
Re}{\kappa}_j^2\leq |\alpha_n|^2$, then the following estimates hold:
\begin{align*}
&|E_{ln}(b_j)|\leq\big|E_{ln}(b_j')\big|e^{-d_j(|\alpha_n|^2-{\rm
Re}{\kappa}_j^2)^{1/2}},\quad l=1, 2,\\
&|E_{3n}(b_j)|\leq\frac{1}{|\beta_{jn}|}\big|\alpha_{1n}E_{1n}(b_j')+\alpha_{2n}E_{
2n}(b_j')\big|e^{-d_j(|\alpha_n|^2-{\rm Re}{\kappa}_j^2)^{1/2}},\\
&\abs{\alpha_{1n}E_{2n}(b_j)-\alpha_{2n}E_{1n}(b_j)}\le\abs{\alpha_{1n}E_{2n}
(b_j')-\alpha_{2n}E_{1n}(b_j')}e^{-d_j(|\alpha_n|^2-{\rm
Re}{\kappa}_j^2)^{1/2}}.
\end{align*}
\end{lemma}

\begin{proof}
Since $\mu$ and $\varepsilon$ are constants in $\Omega_1^{'}$ and
$\Omega_2^{'}$, the Maxwell equation \eqref{mee} reduces to the Helmholtz
equations:
\begin{equation}\label{he}
\Delta E_l+\kappa^2_j E_l=0 \quad\text{in}\ \Omega_j^{'},\, l=1,2,3.
\end{equation}
Plugging \eqref{Ej} into \eqref{he}, we derive the second order ordinary
differential equation for the Fourier coefficient $E_{ln}$:
\begin{equation}\label{ode}
E_{ln}^{''}(x_3)+(\kappa^2_j-|\alpha_n|^2)E_{ln}(x_3)=0,\quad
x_3\not\in (b_2',b_1').
\end{equation}
The general solution of \eqref{ode} is
\[
E_{ln}(x_3)=A_{ln}e^{{\rm i}\beta_{jn}x_3}+B_{ln}e^{-{\rm
i}\beta_{jn}x_3}.
\]
Noting \eqref{beta} and using the radiation condition, we obtain
\begin{equation}\label{Ejz}
\begin{cases}
E_{ln}(x_3)=E_{ln}(b_1')e^{{\rm i}\beta_{1n}(x_3-b_1')}, &\quad
x_3>b_1',\\
E_{ln}(x_3)=E_{ln}(b_2')e^{-{\rm i}\beta_{2n}(x_3-b_2')}, &\quad
x_3<b_2',
\end{cases}
\end{equation}
which gives that
\begin{align*}
|E_{ln}(b_j)|&=|E_{ln}(b_j')| e^{-d_j{\rm Im}\beta_{jn}},\quad
l=1,2,\\
\abs{\alpha_{1n}E_{2n}(b_j)-\alpha_{2n}E_{1n}(b_j)}
&=\abs{\alpha_{1n}E_{2n}(b_j')-\alpha_{2n}E_{1n}(b_j')}e^{-d_j{\rm
Im}\beta_{jn}}.
\end{align*}

On the other hand, it follows from the divergence free condition that
\[
\partial_{x_1} E_1+\partial_{x_2}E_2+\partial_{x_3}E_3=0\quad\text{in}\
\Omega_j',
\]
which yields
\begin{equation}\label{dfc}
E_{3n}'(b_j')+{\rm i}\alpha_{1n}E_{1n}(b_j')+{\rm
i}\alpha_{2n}E_{2n}(b_j')=0.
\end{equation}
It follows from \eqref{Ejz} that
\[
E_{3n}'(x_3)={\rm
i}\beta_{1n}E_{3n}(b_1')e^{{\rm i}\beta_{1n}(x_3-b_1')},\quad x_3> b_1'.
\]
Evaluating the above equation at $x_3=b_1'$ and using the divergence free
condition \eqref{dfc}, we have
\[
E_{3n}(b_1')=-\frac{\alpha_{1n}}{\beta_{1n}}E_{1n}(b_1')-\frac{\alpha_{
2n}}{\beta_{1n}}E_{2n}(b_1'),
\]
which implies
\[
E_{3n}(x_3)=\frac{-1}{\beta_{1n}}\big(\alpha_{1n}E_{1n}(b_1')+\alpha_{2n}E_{
2n}(b_1')\big)e^{{\rm i}\beta_{1n}(x_3-b_1')}.
\]
Taking $x_3=b_1$ in the above equation yields
\[
|E_{3n}(b_1)|= \frac{1}{|\beta_{1n}|}\abs{\alpha_{1n}E_{1n}(b_1')
  +\alpha_{2n}E_{2n}(b_1')}e^{-d_1{\rm Im} \beta_{1n}}.
\]
Similarly, we can obtain
\[
|E_{3n}(b_2)|= \frac{1}{|\beta_{2n}|}\abs{\alpha_{1n}E_{1n}(b_2')
  +\alpha_{2n}E_{2n}(b_2')}e^{-d_2{\rm Im} \beta_{2n}}.
\]

Using \eqref{beta} again, we have
\eqn{
  {\rm Im}\beta_{jn}^2={\rm
Im}(\kappa_j^2-|\alpha_n|^2)
  ={\rm Im}\kappa_j^2=\omega^2\mu_j{\rm Im}\varepsilon_j\geq 0
}
and
\eqn{
{\rm Im}\beta_{jn}&=\frac{1}{\sqrt{2}}\left[\Big(\big({\rm
Im}\beta_{jn}^2\big)^2+\big({\rm Re}\beta_{jn}^2\big)^2\Big)^{1/2}-{\rm
Re}\beta_{jn}^2\right]^{1/2}\\
&\geq(-{\rm
Re}\beta_{jn}^2)^{1/2}=(\alpha_n^2-{\rm Re}\kappa_j^2)^{1/2},
}
which completes the proof.
\end{proof}

Noting that $\Omega$ is a cuboid, we have the following Hodge decomposition: For
any $\bm{\psi}\in H_{\text{qper}}(\text{curl},\Omega)$,
there exist $\bm{\psi}^{(1)}\in H^1(\Omega)^3$ and
$\psi^{(2)}\in H^1(\Omega)$ such that
\[
\bm{\psi}=\bm{\psi}^{(1)}+\nabla\psi^{(2)},\quad\nabla\cdot\bm{\psi}^{(1)}=0.
\]

Let $\Pi_h: H^1(\Omega)\rightarrow U_h$
be the Scott--Zhang interpolation operator. For any element
$T\in\mathcal{M}_h$ with the size of $h_T$
and any face $F\in\mathcal{F}_h$ with the size of $h_F$, one has
\begin{align*}
  \|\psi^{(2)}-\Pi_h\psi^{(2)}\|_{L^2(T)}
  &\leq Ch_{T}|\psi^{(2)}|_{H^1(\tilde{T})},\\
  \|\psi^{(2)}-\Pi_h\psi^{(2)}\|_{L^2(F)}
  &\leq Ch_{F}^{1/2}|\psi^{(2)}|_{H^1(\tilde{F})}.
\end{align*}
Here $\tilde{T}$ and $\tilde{F}$ are the union of all the elements in
$\mathcal{M}_h$, which have nonempty
intersection with the element $T$ and the face $F$, respectively.

\begin{lemma}
  There exists a linear projection operator
  $\mathscr{P}_h: H_{{\rm qper}}({\rm curl},
\Omega)\cap H^1(\Omega)^3\rightarrow V_h$
  satisfying the following estimates:
  \begin{align*}
    &\|\mathscr{P}_h\bm{\psi}^{(1)}\|_{L^2(T)^3}\leq
C(\|\bm{\psi}^{(1)}\|_{L^2(\tilde{T})^3}
    +h_T{|\bm{\psi}^{(1)}|}_{H^1(\tilde{T})^3}),\\
    &\|\bm{\psi}^{(1)}-\mathscr{P}_h\bm{\psi}^{(1)}\|_{L^2(T)^3}
    \leq Ch_T\|\bm{\psi}^{(1)}\|_{H^1(\tilde{T})^3},\\
    &\|\bm{\psi}^{(1)}-\mathscr{P}_h\bm{\psi}^{(1)}\|_{L^2(F)^3}
    \leq Ch_F^{1/2}\|\bm{\psi}^{(1)}\|_{H^1(\tilde{F})^3}.
\end{align*}
\end{lemma}

Define $\bm{\psi}_h^{(1)}={\mathscr{P}_h}\bm{\psi}^{(1)},
\psi_h^{(2)}={\Pi_h}\psi^{(2)},
\bm{\psi}_h=\bm{\psi}_h^{(1)}+\nabla\psi_h^ {(2)}$.

\begin{lemma}\label{att}
There exists an integer $N_{j1}$ independent of $h$ and satisfying
  {$\big(\frac{2\pi N_{j1}}{\LL}\big)^2$ $>{\rm Re}\kappa_j^2$}, $j=1,2$
  such that for any $N_j\geq N_{j1}$ and
  $\bm{\psi}\in H_{{\rm qper}}({\rm curl},\Omega)$ the following estimate holds:
  \begin{align*}
  &\Big|a(\bm{\xi}, \bm{\psi})+{\rm i}\omega\sum_{j=1}^2\int_{\Gamma_j}
  (\mathscr{T}_j-\mathscr{T}_j^{N_j})
  \bm{\xi}_{\Gamma_j}\cdot\bar{\bm\psi}_{\Gamma_j}\Big|\\
  \leq& C_1\Bigg(\bigg(\sum_{T\in\mathcal{M}_h}\eta_T^2\bigg)^{1/2}+\sum_{j=1}^2 e^{-d_j\si_j}\|\bm{E}^{\rm
inc}\|_{TL^2(\Gamma_1)}\Bigg)\|\bm{\psi}\|_{H({\rm curl},\Omega)},
  \end{align*}
where $\si_j=\big(\big(\frac{2\pi N_j}{\LL}\big)^2-{\rm
Re}\kappa_j^2\big)^{1/2}$. Moreover,
\eqn{|(\ep\bm{\xi},\na q)|\ls\bigg(\sum_{T\in\mathcal{M}_h}\eta_T^2\bigg)^{1/2}\norm{\na q}_{L^2(\Om)^2}.}
\end{lemma}

\begin{proof}
The second estimate is a direct consequence of \eqref{xiq} with
$q_h=\Pi_hq$. It remains to prove the first estimate. Define
  \begin{align*}
    J_1&=\langle\bm{f},
\bm{\psi}-\bm{\psi}_h\rangle_{\Gamma_1}-a_N(\bm{E}_h^N,
\bm{\psi}-\bm{\psi}_h),\\
    J_2&={\rm i}\omega\sum_{j=1}^2\int_{\Gamma_j}
    (\mathscr{T}_j-\mathscr{T}_j^{N_j})
    \bm{E}_{\Gamma_j}\cdot\bar{\bm\psi}_{\Gamma_j}.
  \end{align*}
A simple calculation yields that
\[
a(\bm{\xi}, \bm{\psi})+{\rm
i}\omega\sum_{j=1}^2\int_{\Gamma_j}(\mathscr{T}_j-\mathscr{T}_j^{N_j})\bm{\xi}_{
\Gamma_j}\cdot\bar{\bm\psi}_{\Gamma_j}:= J_1+J_2,
  \]
where
\begin{align*}
  J_1=&\langle\bm{f}, \bm{\psi}-\bm{\psi}_h\rangle_{\Gamma_1}-a_N(\bm{E}_h^N,
\bm{\psi}-\bm{\psi}_h)\\
  =&-\int_{\Omega}\mu^{-1}(\nabla\times\bm{E}_h^N)\cdot\nabla\times
(\bar{\bm\psi}-\bar{\bm\psi}_h)+\omega^2\int_{\Omega}\varepsilon\bm{E}
_h^N\cdot(\bar{\bm\psi}-\bar{\bm\psi}_h )\\
&+{\rm
i}\omega\sum_{j=1}^2\int_{\Gamma_j}\mathscr{T}_j^{N_j}(\bm{E}_h^N)_{\Gamma_j
}\cdot(\bar{\bm\psi}-\bar{\bm\psi}_h)_{\Gamma_j}+\langle\bm{f},
\bm{\psi}-\bm{\psi}_h\rangle_{\Gamma_1}\\
=&J_1^1+J_1^2+J_1^3+J_1^4.
\end{align*}
Using the fact $\nabla\times \nabla(\bar{\psi}^{(2)}-\bar{\psi}_h^{(2)})=0$
and Green's theorem yields
\begin{align*}
  J_1^1=&-\int_{\Omega}\mu^{-1}(\nabla\times\bm{E}_h^N)\cdot\nabla\times
  (\bar{\bm\psi}-\bar{\bm\psi}_h)\\
       =&-\int_{\Omega}\mu^{-1}(\nabla\times\bm{E}_h^N)\cdot\nabla\times
(\bar{\bm\psi}^{(1)}-\bar{\bm\psi}_h^{(1)}+\nabla\bar{\psi}^{(2)}
-\nabla\bar{\psi}_h^{(2)})\\
=&-\int_{\Omega}\mu^{-1}(\nabla\times\bm{E}_h^N)\cdot\nabla\times
     (\bar{\bm\psi}^{(1)}-\bar{\bm\psi}_h^{(1)})\\
  =&-\sum_{T\in\mathcal{M}_h}\int_T \mu^{-1}
  (\nabla\times\bm{E}_h^N)\cdot\nabla\times
     (\bar{\bm\psi}^{(1)}-\bar{\bm\psi}_h^{(1)})\\
  =&-\sum\limits_{T\in\mathcal{M}_h}\bigg[\int_T\nabla\times
  (\mu^{-1}\nabla\times\bm{E}_h^N)\cdot
     (\bar{\bm\psi}^{(1)}-\bar{\bm\psi}_h^{(1)})\\
&\hspace{1cm}+\int_{\partial
T}(\mu^{-1}\nabla\times\bm{E}_h^N)\times\bm{\nu}\cdot
     (\bar{\bm\psi}^{(1)}-\bar{\bm\psi}_h^{(1)})\bigg].
\end{align*}
It follows from Green's formula that we have
\begin{align*}
J_1^2=&\omega^2\int_{\Omega}\varepsilon\bm{E}_h^N\cdot(\bar{\bm\psi}-\bar{
\bm\psi}_h)\\
  =&\sum_{T\in\mathcal{M}_h}\omega^2\int_T \varepsilon\bm{E}_h^N
  \cdot(\bar{\bm\psi}-\bar{\bm\psi}_h)\\
=&\sum\limits_{T\in\mathcal{M}_h}\bigg[\omega^2\int_T
\varepsilon\bm{E}_h^N\cdot(\bar{\bm\psi}^{(1)}-\bar{\bm\psi}_h^{(1)})
  +\omega^2\int_T \varepsilon\bm{E}_h^N
  \cdot(\nabla\bar{\psi}^{(2)}-\nabla\bar{\psi}_h^{(2)})\bigg]\\
=&\sum_{T\in\mathcal{M}_h}\bigg[\omega^2\int_T \varepsilon\bm{E}_h^N
  \cdot(\bar{\bm\psi}^{(1)}-\bar{\bm\psi}_h^{(1)})
  -\omega^2\int_T\nabla\cdot(\varepsilon\bm{E}_h^N)
  (\bar{\psi}^{(2)}-\bar{\psi}_h^{(2)})\\
&\hspace{1cm}+\omega^2\int_{\partial T}(\varepsilon\bm{E}_h^N\cdot\bm{\nu})
  (\bar{\psi}^{(2)}-\bar{\psi}_h^{(2)})\bigg].
\end{align*}
 Since $\varepsilon, \mu$ are biperiodic functions and
$\bm{E}_h^N, \bm{\psi}^{(1)}-\bm{\psi}^{(1)}_h, \psi^{(2)}-\psi^{(2)}_h$
are quasi-biperiodic functions, it is easy to verify that
\[
  \sum_{j=1}^2\sum_{T\in\mathcal{M}_h}\bigg[
  \int_{\partial T\cap\Gamma_{j0}} (\varepsilon\bm{E}_h^N\cdot\bm{\nu})
  (\bar{\psi}^{(2)}-\bar{\psi}^{(2)}_h)
  +\int_{\partial
T'\cap\Gamma_{j1}} (\varepsilon\bm{E}_h^N\cdot\bm{\nu})
  (\bar{\psi}^{(2)}-\bar{\psi}^{(2)}_h)\bigg]=0,
\]
where $T'$ is the tetrahedron having one of its faces on $\Gamma_{j1}$
corresponding to $T$. Again using Green's formula, we have
\begin{align*}
J_1^3=&{\rm
i}\omega\sum_{j=1}^2\int_{\Gamma_j}\mathscr{T}_j^{N_j}(\bm{E}_h^N)_{\Gamma_j}
\cdot(\bar{\bm\psi}^{(1)}-\bar{\bm\psi}_h^{(1)})
+{\rm i}\omega\sum_{j=1}^2\int_{\Gamma_j}\mathscr{T}_j^{N_j}(\bm{E}_h^N)_{
\Gamma_j} \cdot\nabla_{\Gamma_j}(\bar{\psi}^{(2)}-\bar{\psi}^{(2)}_h)\\
=&{\rm i}\omega\sum_{j=1}^2\int_{\Gamma_j}\mathscr{T}_j^{N_j}(\bm{E}_h^N)_{
\Gamma_j} \cdot(\bar{\bm\psi}^{(1)}-\bar{\bm\psi}_h^{(1)})\\
&\hspace{1cm}-{\rm i}\omega\sum_{j=1}^2\int_{\Gamma_j}\text{div}_{
\Gamma_j}(\mathscr{T}_j^{N_j}(\bm{E}_h^N)_{\Gamma_j})
  (\bar{\psi}^{(2)}-\bar{\psi}^{(2)}_h)\db\\
=&\sum_{T\in\mathcal{M}_h}\sum_{j=1}^2
  \sum_{F\subset\partial T\cap\Gamma_j}{\rm i}\omega\int_F
  (\mathscr{T}_j^{N_j}(\bm{E}_h^N)_{\Gamma_j})\cdot
  (\bar{\bm\psi}^{(1)}-\bar{\bm\psi}^{(1)}_h)_{\Gamma_j}\\
  &\hspace{1cm}-\sum_{T\in\mathcal{M}_h}\sum_{j=1}^2
  \sum_{F\subset\partial T\cap\Gamma_j}{\rm i}\omega\int_F
  \text{div}_{\Gamma_j}(\mathscr{T}_j^{N_j}(\bm{E}_h^N)_{\Gamma_j})
  (\bar{\psi}^{(2)}-\bar{\psi}^{(2)}_h).
\end{align*}
Applying Green's formula again on $J_1^4$ gives
\begin{align*}
J_1^4=&\langle\bm{f}, \bm{\psi}-\bm{\psi}_h\rangle_{\Gamma_1}=-2{\rm i}
\omega\int_{\Gamma_1} \mathscr{T}_1 \bm{E}^{\rm inc}_{\Gamma_1}\cdot
        (\bar{\bm\psi}-\bar{\bm\psi}_h)_{\Gamma_1}\\
=&-2{\rm i}\omega\int_{\Gamma_1}\mathscr{T}_1 \bm{E}^{\rm inc}_{\Gamma_1}\cdot
        (\bar{\bm\psi}^{(1)}-\bar{\bm\psi}_h^{(1)})_{\Gamma_1}\\
&\hspace{1cm}-2{\rm i}\omega\int_{\Gamma_1}\mathscr{T}_1 \bm{E}^{\rm
inc}_{\Gamma_1} \cdot \nabla_{\Gamma_1}(\bar{\psi}^{(2)}-\bar{\psi}_h^{(2)})\\
=&-2{\rm i}\omega\int_{\Gamma_1}\mathscr{T}_1 \bm{E}^{\rm inc}_{\Gamma_1}\cdot
        (\bar{\bm\psi}^{(1)}-\bar{\bm\psi}_h^{(1)})_{\Gamma_1}\\
&\hspace{1cm}+2{\rm i}\omega\int_{\Gamma_1}\text{div}_{\Gamma_1}
(\mathscr{T}_1
\bm{E}^{\rm inc}_{\Gamma_1})(\bar{\psi}^{(2)}-\bar{\psi}_h^{(2)})\db\\
=&\sum_{T\in\mathcal{M}_h}\sum_{F\subset\partial T\cap\Gamma_1}
       -2{\rm i}\omega\int_{F}\mathscr{T}_1 \bm{E}^{\rm inc}_{\Gamma_1}\cdot
        (\bar{\bm\psi}^{(1)}-\bar{\bm\psi}_h^{(1)})_{\Gamma_1}\\
&\hspace{1cm}+\sum_{T\in\mathcal{M}_h}\sum_{F\subset\partial
T\cap\Gamma_1} 2{\rm i}\omega\int_F\text{div}_{\Gamma_1}
        (\mathscr{T}_1 \bm{E}^{\rm inc}_{\Gamma_1})
        (\bar{\psi}^{(2)}-\bar{\psi}_h^{(2)}).
\end{align*}
Combining the above estimates, we get
\begin{align*}
  &J_1=J_1^1+J_1^2+J_1^3+J_1^4\\
     =&\sum_{T\in\mathcal{M}_h}\bigg\{
     \int_T(\omega^2 \varepsilon\bm{E}_h^N
     -\nabla\times(\mu^{-1}\nabla\times\bm{E}_h^N))
\cdot(\bar{\bm\psi}^{(1)}-\bar{\bm\psi}^{(1)}_h)-\int_T\nabla\cdot
(\omega^2\varepsilon\bm{E}_h^N)(\bar{\psi}^{(2)}-\bar{\psi}_h^{(2)})\\
     &+\sum_{F\in\mathcal{F}_h,F\subset\partial T}
     \Big[-\int_F(\mu^{-1}\nabla\times\bm{E}_h^N)\times\bm{\nu}\cdot
(\bar{\bm\psi}^{(1)}-\bar{\bm\psi}^{(1)}_h)+\int_F \omega^2
\varepsilon\bm{E}_h^N\cdot\bm{\nu}(\bar{\psi}^{(2)}-\bar{\psi}_h^{(2)})\Big]
\bigg\}\\
     &+\sum_{T\in\mathcal{M}_h}\sum_{j=1}^2
     \sum_{F\subset\partial T\cap\Gamma_j}\Big[{\rm i}\omega
     \int_F(\mathscr{T}_j^{N_j}(\bm{E}_h^N)_{\Gamma_j})\cdot
(\bar{\bm\psi}^{(1)}-\bar{\bm\psi}_h^{(1)})_{\Gamma_j}\\
&\hskip 100pt-{\rm i}
\omega\int_F\text{div}_{\Gamma_j}
     (\mathscr{T}_j^{N_j}(\bm{E}_h^N)_{\Gamma_j})
  (\bar{\psi}^{(2)}-\bar{\psi}^{(2)}_h)\Big]\db\\
  &+\sum_{T\in\mathcal{M}_h}\sum_{F\subset\partial T\cap\Gamma_1}
       -2{\rm i}\omega\int_{F}\mathscr{T}_1 \bm{E}^{\rm inc}_{\Gamma_1}\cdot
        (\bar{\bm\psi}^{(1)}-\bar{\bm\psi}_h^{(1)})_{\Gamma_1}\db\\
        &+\sum_{T\in\mathcal{M}_h}\sum_{F\subset\partial T\cap\Gamma_1}
        2{\rm i}\omega\int_{F}\text{div}_{\Gamma_1}
        (\mathscr{T}_1 \bm{E}^{\rm
inc}_{\Gamma_1})(\bar{\psi}^{(2)}-\bar{\psi}_h^{(2)}).
 \end{align*}
 Using the residuals, we have
 \begin{align*}
    J_1=&\sum_{T\in\mathcal{M}_h}\bigg\{ \int_T R_T^{(1)}\cdot
    (\bar{\bm\psi}^{(1)}-\bar{\bm\psi}_h^{(1)})+\int_T R_T^{(2)}
    (\bar{\psi}^{(2)}-\bar{\psi}_h^{(2)})\\
    &+\sum_{F\in\mathring{F}_h\cap\partial T}
\Big[\int_F\frac{1}{2}J_F^{(1)}\cdot(\bar{\bm\psi}^{(1)}-\bar{\bm\psi}_h^{(1)}
)+\int_F\frac{1}{2}J_F^{(2)}(\bar{\psi}^{(2)}-\bar{\psi}_h^{(2)})\Big]\\
&+\sum_{j=1}^2\sum_{F\subset\partial
T\cap\Gamma_j} \Big[\int_F\frac{1}{2}J_F^{(1)}\cdot(\bar{\bm\psi}^{(1)}-\bar{
\bm\psi}_h^{(1)})_{\Gamma_j}+\int_F\frac{1}{2}J_F^{(2)}(\bar{\psi}^{(2)}-\bar{
\psi}_h^{(2)})\Big] \bigg\}  \\
    =&\sum_{T\in\mathcal{M}_h}\bigg\{ \int_T R_T^{(1)}\cdot
    (\bar{\bm\psi}^{(1)}-\bar{\bm\psi}_h^{(1)})+\int_T R_T^{(2)}
    (\bar{\psi}^{(2)}-\bar{\psi}_h^{(2)})\\
    &+\sum_{F\subset\partial T}
\Big[\int_F\frac{1}{2}J_F^{(1)}\cdot(\bar{\bm\psi}^{(1)}-\bar{\bm\psi}_h^{(1)}
)_{\Gamma_j}+\int_F\frac{1}{2}J_F^{(2)}(\bar{\psi}^{(2)}-\bar{\psi}_h^{(2)})\Big
]\bigg\}.
\end{align*}
Taking $\psi_h^{(1)}=\mathscr{P}_h\psi^{(1)}$ and
$\psi_h^{(2)}=\Pi_h\psi^{(2)}$, we obtain
\begin{align}
|J_1|\leq&C\sum_{T\in\mathcal{M}_h}\bigg\{\Big[h_T\|R_T^{(1)}\|_{L^2(T)^3
}|\bm{\psi}^{(1)}|_{H^1(\tilde{T})^3}+h_T\|R_T^{(2)}\|_{L^2(T)}
  |\psi^{(2)}|_{H^1(\tilde{T})}\Big]\nonumber\\
  &+\sum_{F\in\partial T}\Big[h_F^{1/2}\|J_F^{(1)}\|_{L^2(F)^3}
  |\bm{\psi}^{(1)}|_{H^1(\tilde{F})^3}+h_F^{1/2}\|J_F^{(2)}\|_{L^2(F)}
  |\psi^{(2)}|_{H^1(\tilde{F})}\Big]\bigg\}\nonumber\db\\
  \leq&C\left(\sum_{T\in\mathcal{M}_h}\eta_T^{2}\right)^{1/2}
\left(|\psi^{(1)}|_{H^1(\Omega)^3}^2+|\psi^{(2)}|_{H^1(\Omega)}^2\right)^{1/2}
  \nonumber\\
  \leq &C\left(\sum_{T\in\mathcal{M}_h}\eta_T^{2}\right)^{1/2}
  \|\bm{\psi}\|_{H({\rm curl},\Omega)}.\label{J1}
\end{align}

It remains to estimate $J_2$. A straightforward calculation yields
\begin{align*}
|J_2|&=\omega\bigg|\sum_{j=1}^2\int_{\Gamma_j}(\mathscr{T}_j-\mathscr{T}_j^{N_j}
)\bm{E}_{\Gamma_j}\cdot\bar{\bm\psi}_{\Gamma_j}\bigg|\\
  =&\omega\bigg{|}\sum_{j=1}^2\int_{\Gamma_j}\sum_{n{\not\in} U_{N_j}}
(r_{1n}^{(j)},r_{2n}^{(j)},0)^\top e^{{\rm i}(\alpha_{1n}x_1+\alpha_{2n}x_2)}
\cdot\sum_{n\in\mathbb{Z}^2}(\bar{\psi}_{1n}^{(j)},\bar{\psi}_{2n}^{(j)},
0)^\top e^{-{\rm i}(\alpha_{1n}x_1+\alpha_{2n}x_2)}\bigg|\db\\
  =&\omega L_1L_2\bigg|\sum_{j=1}^2\sum_{n\not\in U_{N_j}}
(r_{1n}^{(j)}\bar{\psi}_{1n}^{(j)}+r_{2n}^{(j)}\bar{\psi}_{2n}^{(j)}
)\bigg|\\
  =&L_1L_2\bigg|\sum_{j=1}^2\sum_{n\not\in U_{N_j}}
  \frac{1}{\mu_j\beta_{jn}}\Big[\omega^2\varepsilon_j\mu_j
({E_{1n}(b_j)}\bar{\psi}_{1n}^{(j)}+E_{2n}(b_j)\bar{\psi}_{2n}^{(j)})\\
  &\hskip 100pt-(\alpha_{1n}E_{2n}(b_j)-\alpha_{2n}E_{1n}(b_j))
(\alpha_{1n}\bar{\psi}_{2n}^{(j)}-\alpha_{2n}\bar{\psi}_{1n}^{(j)})\Big]\bigg|.
  \end{align*}
Let $N_{j1}$ be a sufficiently large
integer such that
\[
\Big(\frac{2\pi N_{j1}}{\LL}\Big)^2>\text{Re}\kappa_j^2, \quad
|\mu_j\beta_{jn}|\gtrsim(|\alpha_n|^2+1)^{1/2}.
\]
Suppose $N_j\geq N_{j1}$. It follows from Lemma \ref{expo deca} that
\begin{align*}
  &|J_2|\leq L_1L_2\sum_{j=1}^2\sum_{n\not\in U_{N_j}}
  \frac{e^{-d_j (|\alpha_n|^2-\text{Re}
  \kappa_j^2)^{1/2}}}{|\mu_j\beta_{jn}|}\Big[|\omega^2\varepsilon_j\mu_j|(|E_{1n}(b_j')||\psi_{1n}(b_j)|\\
  &\quad\quad+|E_{2n} (b_j')||\psi_{2n}
(b_j)|)+|\alpha_{1n}E_{2n}(b_j')-\alpha_{2n}E_{1n}(b_j')|
  |\alpha_{1n}\psi_{2n}(b_j)-\alpha_{2n}\psi_{1n}(b_j)|\Big]\db\\
   &\ls L_1 L_2\sum_{j=1}^2{e^{-d_j \big(\big(\frac{2\pi
N_j}{\LL}\big)^2-\text{Re}\kappa_j^2\big)^{1/2}}}\\
  &\quad\times\bigg\{\sum_{n{\not\in} U_{N_j}}
(|\alpha_n|^2+1)^{-\frac12}\Big[|E_{1n}(b_j')|^2+|E_{2n}(b_j')|^2
+|\alpha_{1n}E_{2n}(b_j')-\alpha_{2n}E_{1n}(b_j')|^2\Big]\bigg\}^{\frac12}\\
  &\quad\times\bigg\{\sum_{n\not\in U_{N_j}}
(|\alpha_n|^2+1)^{-\frac12}[|\psi_{1n}(b_j)|^2+|\psi_{2n}(b_j)|^2
  +|\alpha_{1n}\psi_{2n}(b_j)-\alpha_{2n}\psi_{1n}(b_j)|^2]\bigg\}^{\frac12}\db\\
  &\ls\sum_{j=1}^2 e^{-d_j \big(\big(\frac{2\pi N_j}{\LL}\big)^2-\text{Re}\kappa_j^2\big)^{1/2}}
  \|\bm{E}\|_{TH_{\rm qper}^{-1/2}({\rm curl},\Gamma_j')}
  \|\bm{\psi}\|_{TH_{\rm qper}^{-1/2}({\rm curl}, \Gamma_j)}\\
  &\ls\sum_{j=1}^2 e^{-d_j \big(\big(\frac{2\pi N_j}{\LL}\big)^2-\text{Re}\kappa_j^2\big)^{1/2}}\|\bm{E}\|_{H({\rm
curl},\Omega)} \|\bm{\psi}\|_{H({\rm curl},\Omega)}.
\end{align*}
Using the inf-sup condition \eqref{isc} yields
\[
\|\bm{E}\|_{H({\rm curl},\Omega)}\leq\frac{1}{\gamma_1}\sup_{0\neq\bm{\psi} \in
H_{\rm qper}({\rm
curl}, \Omega)}\frac{|a(\bm{E},\bm{\psi})|}{\|\bm{\psi}\|_{H({\rm
curl},\Omega)}}.
\]
We have from \eqref{vp} that
\begin{align*}
&  |a(\bm{E},\bm{\psi})|=|\langle\bm{f},
\bm{\psi}\rangle_{\Gamma_1}|=|2\omega\int_{\Gamma_1}\mathscr{T}_1 \bm{E}^{\rm
inc}_{\Gamma_1}\cdot\bar{\bm\psi}_{\Gamma_1}|\\
&  \leq C\|\bm{E}^{\rm inc}\|_{TL^2(\Gamma_1)}\|\bm{\psi}\|_{TL^2(\Gamma_1)}
  \leq C\|\bm{E}^{\rm inc}\|_{TL^2(\Gamma_1)}\|\bm{\psi}\|_{H({\rm curl},
\Omega)}.
\end{align*}
Combining the above estimates gives
\begin{equation}\label{J2}
  |J_2|\ls\sum_{j=1}^2e^{-d_j \big(\big(\frac{2\pi N_j}{\LL}\big)^2-\text{Re}\kappa_j^2\big)^{1/2}}
  \|\bm{E}^{\rm inc}\|_{TL^2(\Gamma_1)}\|\bm{\psi}\|_{H({\rm curl},\Omega)},
\end{equation}
The proof is completed by combining \eqref{J1} and \eqref{J2}.
\end{proof}

\subsection{Estimates of the DtN operators}

The following lemma gives an estimate of the second term in the right hand
side of \eqref{xi-curl}.
\begin{lemma}\label{imt}
There exists a positive constant $C$ such that
\begin{equation*}
{\rm
Im}\int_{\Gamma_j}\mathscr{T}_j^{N_j}\bm{\psi}_{\Gamma_j}\cdot\bar{\bm\psi}_{
\Gamma_j} \geq -C\|\bm{\psi}\|^2_{H^{-1/2}_{\rm qper}(\Gamma_j)^3},\quad
\bm{\psi}\in TH_{{\rm{qper}}}^{-1/2}({\rm curl},\Gamma_j).
\end{equation*}
\end{lemma}
\begin{proof}
Define
\[
\kappa_j^2=\omega^2\varepsilon_j\mu_j=u_j+{\rm i}v_j,
\]
It follows from $\mu_j>0$, $\text{Re}(\varepsilon_j)>0$,
and $\text{Im}(\varepsilon_j)\geq 0$ that $u_j>0$ and $v_j\geq 0$. Recall
\[
\beta_{jn}^2=\kappa_j^2-|\alpha_n|^2=w_{jn}+{\rm i}v_j,
\]
where
\[
 w_{jn}=\text{Re}(\omega^2\varepsilon_j\mu_j)-|\alpha_n|^2=u_j-|\alpha_n|^2.
\]
It is clear to note that $u_j\geq w_{jn}$. Noting that $\mu_j>0$,
$\text{Re}(\varepsilon_j)>0$, and $\text{Im}(\varepsilon_j)\geq 0$, we get
\[
\beta_{jn}=\gamma_{jn}+{\rm i}\lambda_{jn},
\]
where
\begin{align*}
&\gamma_{jn}=\text{Re}(\beta_{jn})=\frac{1}{\sqrt{2}}
    \left(\sqrt{w_{jn}^2+v_j^2}+w_{jn}\right)^{1/2},\\
&\lambda_{jn}=\text{Im}(\beta_{jn})=\frac{1}{\sqrt{2}}
    \left(\sqrt{w_{jn}^2+v_j^2}-w_{jn}\right)^{1/2}.
\end{align*}
As a quasi-periodic function, $\bm{\psi}_{\Gamma_j}$ has the expansion
\[
\bm{\psi}_{\Gamma_j}(x_1,x_2,b_j)=\sum_{n\in
\mathbb{Z}^2}(\psi_{1n}(b_j), \psi_{2n}(b_j),
0)^\top e^{{\rm i}(\alpha_{1n}x_1+\alpha_{2n}x_2)}.
\]
We have from the definition of the capacity operator $\mathscr{T}_j$ that
\[
\int_{\Gamma_j}\mathscr{T}_j^{N_j}\bm{\psi}_{\Gamma_j}\cdot\bar{\bm{\psi}}_{\Gamma_j}
=\frac{L_1L_2}{\omega\mu_j}\sum\limits_{n\in U_{N_j}}\bigg{[}\frac{\kappa_j^2}{
\beta_{jn}}(|\psi_{1n}|^2+|\psi_{2n}|^2)-\frac{1}{\beta_{jn}}|\alpha_{1n}\psi_{
2n}-\alpha_{2n}\psi_{1n}|^2\bigg{]}.
\]
Taking the imaginary part gives
\begin{align*}
    {\rm Im}\langle\mathscr{T}_j^{N_j}\bm{\psi}, \bm{\psi}\rangle_{\Gamma_j}
    =&\frac{L_1L_2}{\omega\mu_j}\sum_{n\in U_{N_j}}\bigg[
    \frac{\lambda_{jn}}{\gamma_{jn}^2+\lambda_{jn}^2}
    |\alpha_{1n}\psi_{2n}-\alpha_{2n}\psi_{1n}|^2\\
&\qquad\qquad\quad+\frac{v_j\gamma_{jn}-u_j\lambda_{jn}}{\gamma_{jn}^2+\lambda_{
jn}^2}
    (|\psi_{1n}|^2+|\psi_{2n}|^2)\bigg]\\
    \geq&\frac{L_1L_2}{\omega\mu_j}\sum_{n\in U_{N_j}}
    \frac{v_j\gamma_{jn}-u_j\lambda_{jn}}{\gamma_{jn}^2+\lambda_{jn}^2}
    (|\psi_{1n}|^2+|\psi_{2n}|^2)
\end{align*}

To prove the lemma, it is required to estimate
\begin{align*}
    &\frac{1}{\omega\mu_j}\bigg{|}
    \frac{v_j\gamma_{jn}-u_j\lambda_{jn}}{\gamma_{jn}^2+\lambda_{jn}^2}
    (1+|\alpha_n|^2)^{1/2}\bigg{|}\\
    =&\frac{1}{\omega\mu_j}\left[\frac{1+u_j-w_{jn}}{w_{jn}^2+v_j^2}
    \left(v_j^2\frac{\sqrt{w_{jn}^2+v_j^2}+w_{jn}}{2}
    +u_j^2\frac{\sqrt{w_{jn}^2+v_j^2}-w_{jn}}{2}-u_jv_j^2\right)
    \right]^{1/2}.
\end{align*}
Let
\[
G_j(t)=\frac{1+u_j-t}{t^2+v_j^2}\left(v_j^2\frac{\sqrt{t^2+v_j^2}+t}{2}
+u_j^2\frac{\sqrt{t^2+v_j^2}-t}{2}-u_jv_j^2\right).
\]
It can be seen that $G_j(t)$ is a continuous and positive function for $t\leq
u_j$ and $G_j(t)\rightarrow u_j^2$ as $t\rightarrow-\infty$. Thus the function
$G_j(t)$ reaches its maximum at some $t^*$. Therefore, we have
\[
\frac{1}{\omega\mu_j}\bigg|\frac{v_j\gamma_{jn}-u_j\lambda_{jn}}{\gamma_{jn}
^2+\lambda_{jn}^2}(1+\alpha_n^2)^{1/2}\bigg|\leq\frac{\sqrt{G_j(t^*)}}{
\omega\mu_j}:=C.
\]
A simple calculation yields that
\begin{align*}
&\text{Im}\int_{\Gamma_j}\mathscr{T}_j^{N_j}\bm{\psi}_{\Gamma_j}\cdot\bar{
\bm\psi}_{\Gamma_j}\\
&\geq -CL_1L_2\sum_{n\in
U_{N_j}}(1+|\alpha_n|^2)^{-1/2}(|\psi_{1n}|^2+|\psi_{2n}|^2)\\
&\geq-CL_1L_2\sum_{n\in
\mathbb{Z}^2}(1+|\alpha_n|^2)^{-1/2}(|\psi_{1n}|^2+|\psi_{2n}|^2)\\
&=-C\|\bm{\psi}\|_{H^{-1/2}_{\rm qper}(\Gamma_j)^3}^2,
\end{align*}
which completes the proof.
\end{proof}

The following lemma gives an estimate of the last term in \eqref{xi-L2}.
\begin{lemma}\label{lem TT}
  Let $\bm{W}$ be the solution of the dual problem \eqref{dp}. Then there exist
 integers $N_{j2}$ independent of $h$ and satisfying $\big(\frac{2\pi
N_{j2}}{\LL}\big)^2> \mathrm{Re}(\kappa_j^2)$, $j=1,2$ such that for
$N_j\ge N_{j2}$, the following estimate holds:
\eq{\label{eTT}
&\sum_{j=1}^2\bigg|\omega\int_{\Gamma_j}(\mathscr{T}_j-\mathscr{T}_j^{N_j})\bm{\xi}_{\Gamma_j}\cdot\bar{\bm W}_{\Gamma_j}\bigg|\notag\\
&\leq \frac23\abs{(\ep\bm{\xi},\bm{\xi})}+ C \sum_{j=1}^2N_j^{-2}\big(1+d_j^{-4}\big) \|\bm{\xi}\|_{H({\rm curl},\Omega)}^2,
}
where $C$ is a constant independent of $h$ and $N_j$.
\end{lemma}

\begin{proof}
We show that for $j=1,2$ and $\de>0$,
\begin{align}\label{WB}
&\bigg|\omega\int_{\Gamma_j}(\mathscr{T}_j-\mathscr{T}_j^{N_j})
\bm{\xi}_{\Gamma_j}\cdot\bar{\bm W}_{\Gamma_j}\bigg|\notag\\
\le& \frac{(1+\de)^2}{2}\sum_{n\notin
U_{N_1}}|\al_n|^{-2}\bigg|\ep_j^{-1}\int_{\Ga_j} \ze_{3n}'\bar\ze_{3n}\bigg|+C
N_j^{-2}\big(1+d_j^{-4}\big)\norm{\ze}_{H({\rm curl},\Om)}^2.
\end{align}
Following from Lemma \ref{L-div}, we conclude that
\eqn{&\bigg|\sum_{j=1}^2\bigg({\rm i}\omega\int_{\Gamma_j}(\mathscr{T}_j-\mathscr{T}_j^{N_j})
    \bm{\xi}_{\Gamma_j}\cdot\bar{\bm W}_{\Gamma_j}\bigg)\bigg|\notag\\
\le& \frac{(1+\de)^3}{2}\sum_{j=1}^2\Big(|\ep_j|^{-1}\norm{\bm{\ze}}_{L^{2}(\tilde\Om_j)^3}^2+C N_j^{-2}\big(\de^{-1}d_j^{-2}+1+d_j^{-4}\big)\norm{\ze}_{H({\rm curl},\Om)}^2\Big)\\
\le& \frac{(1+\de)^3}{2}\abs{(\ep\bm{\xi},\bm{\xi})}+C \sum_{j=1}^2N_j^{-2}\big(\de^{-2}+1+d_j^{-4}\big)\norm{\xi}_{H({\rm curl},\Om)}^2}
where we have used
$\abs{(\ep^{-1}\bm{\ze},\bm{\ze})}\le\abs{(\ep\bm{\xi},\bm{\xi})}$ and
$\norm{\ze}_{H({\rm curl},\Om)}\ls\norm{\xi}_{H({\rm curl},\Om)}$ (as
consequences of \eqref{zeta}) to derive the last inequality. Then \eqref{eTT}
can be obtained by taking $\de=\big(\frac43\big)^{1/3}-1$.

We shall only prove \eqref{WB} for $j = 1$ since the proof is
similar for $j = 2$. It follows from the definitions of $\mathscr{T}_1$ and
$\mathscr{T}_1^{N_1}$ that
\begin{align}\label{xiW}
&\bigg|\omega\int_{\Gamma_1}(\mathscr{T}_1-\mathscr{T}_1^{N_1})
    \bm{\xi}_{\Gamma_j}\cdot\bar{\bm W}_{\Gamma_1}\bigg|=\bigg|\frac{\omega}{\ep_1}\int_{\Gamma_1}(\mathscr{T}_1-\mathscr{T}_1^{N_1})
    \bm{\ze}_{\Gamma_1}\cdot\bar{\bm W}_{\Gamma_1}\bigg|\notag\\
&=L_1 L_2\sum_{n\notin
U_{N_1}}\bigg|\frac{1}{\ep_1\mu_1\beta_{1n}}\Big[\kappa_1^2\big(\ze_{1n}(b_1)\bar{W}_{1n}
(b_1)+\ze_ {2n}(b_1)\bar{W}_{2n}(b_1)\big)\\
&\quad-\big(\alpha_{1n}\ze_{2n}(b_1)-\alpha_{2n}\ze_{1n}(b_1)\big)\big(\alpha_{1n}
\bar{W}_{2n}(b_1)-\alpha_{2n}\bar{W}_{1n}(b_1)\big)\Big]\bigg|\notag,
  \end{align}
Let  $\tilde{\Omega}_1=\Omega_1'\setminus\Omega_1=\{\bm{x}\in\mathbb R^3:
0<x_1<L_1, \, 0<x_2<L_2,\, b_1'<x_3<b_1\}.$
Next we consider the dual problem in $\tilde{\Omega}_1$ in order to express
$W_{1n}(b_1)$ and $W_{2n}(b_1)$ in $\bm{\ze}$. Since $\varepsilon$ and $\mu$ are
real constants in $\tilde{\Omega}_1$, the dual problem \eqref{dp} can be
rewritten as
\[
\nabla\times\left(\nabla\times \bm{W}\right)
-\omega^2\varepsilon_1\mu_1\bm{W}=\mu_1\bm{\ze}\quad\text{in} ~\tilde{\Omega}_1.
\]
Using the divergence free condition $\nabla\cdot\bm{W}=0$ in $\tilde{\Omega}_1$,
we may reduce the above equation into the Helmholtz equation
\[
\Delta\bm{W}+\kappa_1^2\bm{W}=-\mu_1\bm{\ze}.
\]
Let $\bm{W}=(W_1, W_2, W_3)^\top$. Componentwisely, we have
\[
\Delta W_j+\kappa_1^2 W_j=-\mu_1\ze_j,\quad j=1,2,3.
\]
Since $W_j$ and $\ze_j$ are quasi-biperiodic functions, they have the following Fourier
series expansions
\[
W_j=\sum_{n\in\mathbb{Z}^2}W_{jn}e^{{\rm i} (\alpha_{1n}x_1+\alpha_{2n}x_2)},\quad \ze_j=\sum_{n\in\mathbb{Z}^2}\ze_{jn}e^{{\rm i} (\alpha_{1n}x_1+\alpha_{2n}x_2)}.
\]
A direction calculation yields that the Fourier coefficient $W_{jn}$ with
$n\notin U_{N_1}$ and $j=1,2$ satisfies the following  two-point boundary value
problem of the ordinary differential equations on the interval $(b_1', b_1)$:
\begin{equation}\label{cdp}
 \begin{cases}
 &W_{jn}^{''}(x_3)-|\beta_{1n}|^2 W_{jn}(x_3)=-\mu_1\ze_{jn}(x_3),\\
&W_{jn}(b_1')=W_{jn}(b_1'),\\
&W_{jn}^{'}(b_1)+|\beta_{1n}| W_{jn}(b_1)=-{\rm
i}\kappa_1^{-2}\mu_1\al_{jn}\ze_{3n}(b_1).
\end{cases}
\end{equation}
Here we have used $W_{3n}'=-{\rm i}\al_{1n}W_{1n}-{\rm i}\al_{2n}W_{2n}$ (as
a consequence of $\nabla\cdot\bm{W}=0$) and $W_{3n}^{''}-|\beta_{1n}|^2
W_{3n}=-\mu_1\ze_{jn}$ to derive the boundary conditions. It is easy to verify
that the solutions to \eqref{cdp} can be expressed as
\begin{align*}
W_{jn}(x_3)=& \frac{\mu_1}{2|\beta_{1n}|}\bigg(-\int_{b_{1}}^{x_{3}}e^{|\beta_{1n}|(x_{3}-s)}\ze_{jn}(s)ds+\int_{b_{1}'}^{x_{3}}e^{|\beta_{1n}|(s-x_{3})}\ze_{jn}(s)ds\\
&\quad
-\int_{b_{1}'}^{b_{1}}e^{|\beta_{1n}|(2b_{1}'-x_{3}-s)}\ze_{jn}(s)ds\bigg)+
e^{|\beta_{1n}|(b_{1}'-x_{3})}W_{jn}(b_1')\\
& \quad - \frac{{\rm i}\mu_1\al_{jn}\ze_{3n}(b_1)}{2\kappa_1^|
\beta_{1n}|} \Big(e^{|\beta_{1n}|(x_3-b_1)}-e^{|\beta_{1n}|(2b_{1}'-b_1-x_3)}
\Big)
\end{align*}
which leads to
\[
W_{jn}(b_1)=\omega_{jn}^I+\omega_{jn}^{II},
\]
where
\begin{align}
\omega_{jn}^I=&\frac{\mu_1}{2|\beta_{1n}|}\bigg(\int_{b_{1}'}^{b_1}e^{|\beta_{1n
}|(s-b_1)}\ze_{jn}(s)ds
-\int_{b_{1}'}^{b_{1}}e^{|\beta_{1n}|(2b_{1}'-b_1-s)}\ze_{jn}(s)ds\bigg)\\
&+ e^{-d_1|\beta_{1n}|}W_{jn}(b_1')\notag\\
\omega_{jn}^{II}=& - \frac{{\rm i}\mu_1\al_{jn}\ze_{3n}(b_1)}{2\kappa_1^2|
\beta_{1n}|} \big(1-e^{-2d_1|\beta_{1n}|}\big). \label{ewjnII}
\end{align}
Denote by $B_1=[b_1',b_1]$. Clearly,
\begin{align}\label{wjnI}
|\omega_{jn}^I| \leq &\frac{\mu_1}{2|\beta_{1n}|^{2}}\|\ze_{jn}\|_{L^{\infty}(B_1)}+e^{-d_1| \beta_{1n}|}|W_{jn}(b_1')|,
\end{align}

Next we turn to estimate the terms in \eqref{xiW}.
First, from \eqref{wjnI} we have
\eqn{|\ze_{jn}(b_1)\bar{\omega}_{jn}^I|&\ls \|\ze_{jn}\|_{L^{\infty}(B_1)}\Big(|\beta_{1n}|^{-2}\|\ze_{jn}\|_{L^{\infty}(B_1)}+e^{-d_1| \beta_{1n}|}|W_{jn}(b_1')|\Big)\\
&\ls |\beta_{1n}|^{-2}\|\ze_{jn}\|_{L^{\infty}(B_1)}^2+|\beta_{1n}|^{2}e^{-2d_1|
\beta_{1n}|}|W_{jn}(b_1')|^2.}
 It is easy to show that (see the proof of \cite[Lemma 4.5]{WBLLW15}):
\begin{align}\label{zetajn}
&\|\ze_{jn}\|_{L^{\infty}(B_1)}^{2}
\leq \frac{2}{d_1}\|\ze_{jn}\|_{L^{2}(B_1)}^{2}+2\|\ze_{jn}\|_{L^{2}
(B_1)}\|\ze_{jn}'\|_{L^{2}(B_1)}.
\end{align}
Therefore
\eq{\label{chi_om}|\ze_{jn}(b_1)\bar{\omega}_{jn}^I|\ls& |\beta_{1n}|^{-3}\Big(|\beta_{1n}|\big(|\beta_{1n}|+d_1^{-1}\big)\|\ze_{jn}\|_{L^{2}(B_1)}^2+ \|\ze_{jn}'\|_{L^{2}(B_1)}^{2}\Big)\\
&+|\beta_{1n}|^{2}e^{-2d_1| \beta_{1n}|}|W_{jn}(b_1')|^2.\notag}
Note that we may choose $N_{12}$ such that
\begin{align}\label{beta1n}
|\al_{n}|\gtrsim |\be_{1n}|\ge (1+\de)^{-1}|\al_{n}|\gtrsim \max(\kappa_1,N_1)
\quad\text{for } n\notin U_{N_1}.
\end{align}
{Simple calculations show that
\eq{\label{ecurl}
\norm{\na\times\bm{\ze}}_{L^2(\Om)^3}^2=&L_1L_2\sum_{n\in\mathbb{Z}^2}
\Big(\norm{\ze_{1n}'-{\rm i}\al_{1n}\ze_{3n}}_{L^{2}([b_2,b_{1}])}^2\\
&+\norm{\ze_{2n}'-{\rm i}\al_{2n}\ze_{3n}}_{L^{2}([b_2,b_{1}])}^2+\norm{\al_{1n}\ze_{2n}-\al_{2n}\ze_{1n}}_{L^{2}([b_2,b_{1}])}^2\Big).\notag
}
From \eqref{chi_om}--\eqref{ecurl},  \eqref{normcurlgamma}, Lemma~\ref{tt}, and \eqref{dps}, we conclude that
\begin{align}\label{xiW1}
&L_1 L_2\bigg|\sum_{n\notin
U_{N_1}}\frac{\kappa_1^2}{\ep_1\mu_1\beta_{1n}}\big(\ze_{1n}(b_1)\bar{\omega}_{
1n}^I+\ze_ {2n}(b_1)\bar{\omega}_{2n}^I\big)\bigg|\notag\\
&\ls \sum_{n\notin U_{N_1}}\sum_{j=1}^2\Big[ |\beta_{1n}|^{-4}\Big(|\beta_{1n}|\big(|\beta_{1n}|+d_1^{-1}\big)\|\ze_{jn}\|_{L^{2}(B_1)}^2+ \|\ze_{jn}'\|_{L^{2}(B_1)}^{2}\Big)\notag\\
&\quad +|\beta_{1n}|^2e^{-2d_1|
\beta_{1n}|}(1+|\alpha_{n}|^2)^{-1/2}|W_{jn}(b_1')|^2\Big].\notag\\
&\ls \sum_{n\notin U_{N_1}}\sum_{j=1}^2\Big[ |\al_n|^{-4}\Big( |\al_n|\big(|\al_n|+d_1^{-1}\big)\|\ze_{jn}\|_{L^{2}(B_1)}^2+\|\ze_{jn}'-{\rm i}\al_{jn}\ze_{3n}\|_{L^{2}(B_1)}^{2}\notag\\
&\quad+|\al_{jn}|^2\|\ze_{3n}\|_{L^{2}(B_1)}^2\Big)\Big]+|\beta_{1n}|^2e^{
-2d_1|\beta_{1n}|}\|\bm{W}\|^2_{TH_{\text{qper}}^{-1/2}(\text{curl},\Gamma_1')}
\notag\\
&\ls N_1^{-2}\Big(\big(1+(N_1d_1)^{-1}\big)\norm{\ze}_{H({\rm curl},\tilde\Om_1)}^2+|\beta_{1n}|^4e^{-2d_1| \beta_{1n}|}\|\bm{W}\|^2_{H(\text{curl},\Om)}\Big)\notag\\
&\ls N_1^{-2}\big(1+d_1^{-4}\big)\norm{\ze}_{H({\rm curl},\Om)}^2,
\end{align}
where we have used $\max_{s\ge 0}(s^4e^{-2d_1s})\ls d_1^{-4}$ to derive the last
inequality.}

Following from \eqref{ewjnII}, ${\rm div}\bm{\ze}=0$, and \eqref{beta1n},
we conclude that
\begin{align}\label{xiW2}
&L_1 L_2\bigg|\sum_{n\notin
U_{N_1}}\frac{\kappa_1^2}{\ep_1\mu_1\beta_{1n}}\big(\ze_{1n}(b_1)\bar{\omega}_{
1n}^{II}+\ze_ {2n}(b_1)\bar{\omega}_{2n}^{II}\big)\bigg|\notag\\
=&L_1 L_2\bigg|\sum_{n\notin
U_{N_1}}\big(-{\rm i}\al_{1n}\ze_{1n}(b_1)-{\rm i}\al_{2n}\ze_{2n}(b_1)\big)\bar\ze_{3n}(b_1) \frac{1-e^{-2d_1|\beta_{1n}|}}{2\ep_1| \beta_{1n}|^2}\bigg|\notag\\
=&L_1 L_2\bigg|\sum_{n\notin
U_{N_1}}\ze_{3n}'(b_1)\bar\ze_{3n}(b_1) \frac{1-e^{-2d_1|\beta_{1n}|}}{2\ep_1| \beta_{1n}|^2}\bigg|\notag\\
=&\bigg|\sum_{n\notin
U_{N_1}}\frac{1-e^{-2d_1|\beta_{1n}|}}{2| \beta_{1n}|^2}\ep_1^{-1}\int_{\Ga_1}\ze_{3n}'\bar\ze_{3n}\bigg|\notag\\
\le&\frac{(1+\de)^2}{2}\sum_{n\notin
U_{N_1}}|\al_n|^{-2}\bigg|\ep_1^{-1}\int_{\Ga_1}\ze_{3n}'\bar\ze_{3n}\bigg|.
\end{align}

Denote by $V_n:=\alpha_{1n}W_{2n}-\alpha_{2n}W_{1n}$ and
$\tau_n:=\alpha_{1n}\ze_{2n}(x_3)-\alpha_{2n}\ze_{1n}(x_3)$. From \eqref{cdp},
we have
\begin{equation}\label{cdp2}
 \begin{cases}
 &V_n^{''}(x_3)-|\beta_{1n}|^2 V_n(x_3)=-\mu_1\tau_n,\\
&V_n(b_1')=V_n(b_1'),\\
&V_n^{'}(b_1)+|\beta_{1n}| V_n(b_1)=0.
\end{cases}
\end{equation}
Similarly, we may obtain the solution of \eqref{cdp2}
\begin{align*}
|V_n(b_{1})|
\leq \frac{\mu_1}{2|\beta_{1n}|^{2}}\|\tau_{n}\|_{L^{\infty}(B_1)}+e^{-d|
\beta_{1n}|}|V_n(b_1')|,
\end{align*}
which implies by combining with \eqref{beta1n}, \eqref{zetajn}, \eqref{ecurl},
{\eqref{normcurlgamma}, Lemma~\ref{tt}, and \eqref{dps} that
\begin{align}\label{xiW3}
&L_1 L_2\sum_{n\notin
U_{N_1}}\bigg|\frac{\big(\alpha_{1n}\ze_{2n}(b_1)-\alpha_{2n}\ze_{1n}(b_1)\big)}{\ep_1\mu_1\beta_{1n}}\big(\alpha_{1n}
\bar{W}_{2n}(b_1)-\alpha_{2n}\bar{W}_{1n}(b_1)\big)\Big]\bigg|\notag\\
&=L_1 L_2\bigg|\sum_{n\notin
U_{N_1}}\frac{1}{\ep_1\mu_1\beta_{1n}}\tau_{n}(b_1)\cdot\bar V_n(b_1)|\bigg|\notag\\
&\ls\sum_{n\notin U_{N_1}}\Big(|\beta_{1n}|^{-3}\norm{\tau_n}_{L^{\infty}(B_1)}^{2}+|\beta_{1n}|e^{-2d|\beta_{1n}|}|V_n(b_1')|^2\Big)\notag\\
&\ls \sum_{n\notin U_{N_1}}\Big[ |\beta_{1n}|^{-4}\Big(|\beta_{1n}|\big(|\beta_{1n}|+d_1^{-1}\big)\|\tau_n\|_{L^{2}(B_1)}^2+ \|\tau_n'\|_{L^{2}(B_1)}^{2}\Big)\notag\\
&\hskip 60pt+|\beta_{1n}|^2e^{-2d_1| \beta_{1n}|}(1+|\alpha_{n}|^2)^{-1/2}|V_n(b_1')|^2\Big].\notag\db\\
&\ls \sum_{n\notin U_{N_1}}\Big[|\al_n|^{-4}\Big(|\al_n|\big(|\al_n|+d_1^{-1}\big)\|\tau_n\|_{L^{2}(B_1)}^2\notag\\
&\hskip 60pt+ \|\alpha_{1n}(\ze_{2n}'-{\rm i}\al_{2n}\ze_{3n})-\alpha_{2n}(\ze_{1n}'-{\rm i}\al_{1n}\ze_{3n})\|_{L^{2}(B_1)}^{2}\Big)\notag\\
&\hskip 60pt+|\beta_{1n}|^2e^{-2d_1| \beta_{1n}|}\|\bm{W}\|^2_{TH_{\text{qper}}^{-1/2}(\text{curl},\Gamma_1')}\Big]\notag\db\\
&\ls N_1^{-2}\Big(\big(1+(N_1d_1)^{-1}\big)\norm{\ze}_{H({\rm curl},\tilde\Om_1)}^2+|\beta_{1n}|^4e^{-2d_1| \beta_{1n}|}\|\bm{W}\|^2_{H(\text{curl},\Om)}\Big)\notag\db\\
&\ls N_1^{-2}\big(1+d_1^{-4}\big)\norm{\ze}_{H({\rm curl},\Om)}^2.
\end{align}}

Plugging \eqref{xiW1}, \eqref{xiW2}, and \eqref{xiW3} into
\eqref{xiW}, we arrive at \eqref{WB}. This completes the proof of the lemma.
\end{proof}

\subsection{Proof of Theorem \ref{mt}}

Let $N_j\geq\max(N_{j1},N_{j2})$, $j=1,2$. First, it follows from the error representation formula \eqref{xi-curl}, Lemma~\ref{att}, Lemma~\ref{imt}, and Lemma~\ref{tr} that
\begin{align*}
  \|\bm{\xi}\|_{H({\rm curl},\Omega)}^2
  \leq& C\Bigg(\bigg(\sum_{T\in\mathcal{M}_h}\eta_T^2\bigg)^{1/2}
  +\sum_{j=1}^2e^{-d_j\si_j}\|\bm{E}^{\rm inc}\|_{TL^2(\Gamma_1)}\Bigg)
  \|\bm{\xi}\|_{H({\rm curl},\Omega)}\\
  &+\de\|\nabla\times\bm{\xi}\|^2_{L^2(\Omega)^3}+C(\de)\|\bm{\xi}\|^2_{
L^2(\Omega)^3}+C\|\bm{\xi}\|_{L^2(\Omega)^3}^2.
\end{align*}
which gives after taking $\de=1/2$  that
\begin{align}\label{t1}
  \|\bm{\xi}\|_{H({\rm curl},\Omega)}^2
  \ls& \Bigg(\bigg(\sum_{T\in\mathcal{M}_h}\eta_T^2\bigg)^{1/2}
  +\sum_{j=1}^2e^{-d_j\si_j}\|\bm{E}^{\rm inc}\|_{TL^2(\Gamma_1)}\Bigg)^2+\|\bm{\xi}\|_{L^2(\Omega)^3}^2.
\end{align}
Using \eqref{xi-L2}, Lemma~\ref{att}, \eqref{zeta}, \eqref{dps}, and Lemma \ref{lem
TT}, we obtain
\begin{align*}
| (\ep\bm{\xi},\bm{\xi})|\leq&
C\Bigg(\bigg(\sum_{T\in\mathcal{M}_h}\eta_T^2\bigg)^{1/2}
  +\sum_{j=1}^2e^{-d_j\si_j}\|\bm{E}^{\rm inc}\|_{TL^2(\Gamma_1)}\Bigg)
  \|\bm{\xi}\|_{L^2{(\Omega)}^3}\\
  &+\frac23\abs{(\ep\bm{\xi},\bm{\xi})}+ C \sum_{j=1}^2N_j^{-2}\big(1+d_j^{-4}\big) \|\bm{\xi}\|_{H({\rm curl},\Omega)}^2,
\end{align*}
which implies that
\begin{align*}
\|\bm{\xi}\|_{L^2(\Omega)^3}^2\ls
\Bigg(\bigg(\sum_{T\in\mathcal{M}_h}\eta_T^2\bigg)^{1/2}
  +\sum_{j=1}^2e^{-d_j\si_j}\|\bm{E}^{\rm inc}\|_{TL^2(\Gamma_1)}\Bigg)^2\\
 + C \sum_{j=1}^2N_j^{-2}\big(1+d_j^{-4}\big) \|\bm{\xi}\|_{H({\rm
curl},\Omega)}^2
\end{align*}
The proof is completed by combining the above estimate and \eqref{t1}.

\section{Numerical experiments}\label{Numerical exp}

In this section, we report two examples to demonstrate the competitiveness of
our method. The implementation of the adaptive algorithm is based on parallel
hierarchical grid (PHG) \cite{phg},
which is a toolbox for developing parallel adaptive finite element programs on unstructured
tetrahedral meshes.  The first-order N\'{e}d\'{e}lec's edge element is used in the
numerical tests. The linear system resulted from finite element discretization is
solved by the MUMPS direct solver, which is a general purpose
library for the direct solution of large, sparse systems of linear
equations.
The adaptive FEM algorithm is summarized in Table \ref{alg}.

\begin{table}[ht]
\caption{The adaptive FEM-DtN algorithm.}
\hrulefill
\begin{tabular}{ll}
 1 & Given a tolerance $\epsilon > 0$ and mesh refinement threshold $\tau\in
(0,1)$;\\
 2 & Choose $d_j$ and $\sigma_j$ defined in Theorem \ref{mt} such that $e^{-d_j\sigma_j}< 10^{-8}$;\\
 3 & Construct an initial tetrahedral partition $\mathcal{M}_h$ over $\Omega$ and
compute error estimators;\\
 4 &  While $\epsilon_h>\epsilon$ do\\
 5 & \qquad choose  $\hat{\mathcal{M}}_h\subset\mathcal{M}_h$ according to the
strategy $\eta_{\hat{\mathcal{M}}_h}>\tau\eta_{\mathcal{M}_h}$;\\
 6 & \qquad refine all the elements in $\hat{\mathcal{M}}_h$ and obtain a new
mesh denoted still by $\mathcal{M}_h$;\\
 7 & \qquad solve the discrete problem \eqref{fem} on the new mesh
$\mathcal{M}_h$;\\
 8 & \qquad compute the corresponding error estimators;\\
 9 & End while.
\end{tabular}
\hrulefill
\label{alg}
\end{table}

In the experiments, let $\lambda$, $\theta_1$, $\theta_2$, and
$p=(p_1,p_2,p_3)^\top$ denote the wavelength,
the incident angles, and the polarization of the incident wave, respectively, and let
$n$ denote the refractive index. The examples are computed by both the adaptive DtN algorithm and the adaptive PML
method in \cite{BLW10}.

{\em Example 1}. We consider the simplest biperiodic structure, a plat plane,
where the exact solution is available. We assume that a plane wave $\boldsymbol E^{\rm inc}=\boldsymbol q e^{{\rm
i}(\alpha_1 x_1+\alpha_2 x_2-\beta x_3)}$ is incident on the
plat plane $\{x_3=0\}$, which separates two homogeneous media: $n_1=1$ and $n_2=1.5$. In this example, the parameters are chosen as $\lambda=1 \mu m$, $\theta_1=\pi/6$, $\theta_2=\pi/6$, $p=(-\alpha_2, \alpha_1, 0)^T$. The computational domain
$\Omega=(0,0.5)\times(0,0.5)\times(-0.3, 0.3)$.
The exact solution is as follows:
\begin{align*}
\boldsymbol E=\left\{
\begin{array}{ll}
\boldsymbol p e^{{\rm i}(\alpha_1 x_1+\alpha_2 x_2-\beta_1 x_3)}+ r \boldsymbol p e^{{\rm i}(\alpha_1 x_1+\alpha_2 x_2+\beta_1 x_3)}  &\hbox{ if } x_3\ge 0,\\
t\boldsymbol p e^{{\rm i}(\alpha_1 x_1+\alpha_2 x_2-\beta_2 x_3)}                        &\hbox{ if } x_3 < 0,
\end{array}
\right.
\end{align*}
where $r=(\beta_1-\beta_2)/(\beta_1+\beta_2)$, $t=2\beta_1/(\beta_1+\beta_2)$.

 The mesh and surface plots of the amplitude of the total field
$\boldsymbol E_h^N$ are shown in Figure \ref{ex1:mesh}. The mesh has
446600 tetrahedrons and the total number of degrees of freedom (DoFs) on the mesh
is 1053600. We also present the mesh and surface plots of the amplitude of the
total field $\boldsymbol E_h^N$ obtained by the adaptive PML method in Figure
\ref{ex1pml:mesh}.
Note that the total field $\boldsymbol E_h^N$ is solved, the amplitude in the upper PML is large because of the incident field.
Figure \ref{ex1:err} shows the curves of $\log N_k$ versus $\log
\|\bm{E}-\bm{E}_h^N\|_{H({\rm curl},\Omega)}$,  and the a posteriori error
estimates $\eta_h$, where $N_k$ is the total number of DoFs of the mesh.
It indicates that the meshes and the associated numerical complexity are
quasi-optimal: $\|\bm{E}-\bm{E}_h^N\|_{H({\rm curl},\Omega)}=O(N^{-1/3}_k)$ are valid asymptotically.

\begin{figure}
\centering
\includegraphics[width=0.8\textwidth]{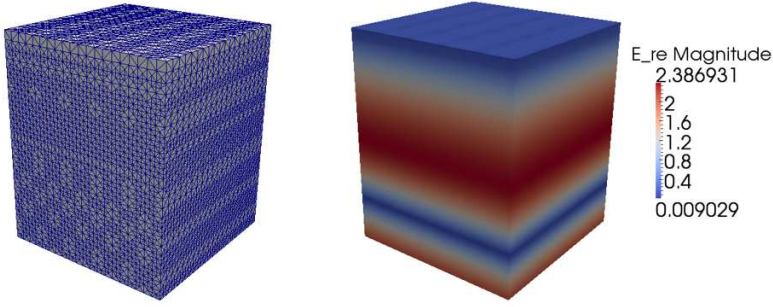}
\caption{The adaptive DtN method: The mesh plot and the surface plot of the
amplitude of the field $\boldsymbol E_h$ after 9 adaptive iterations for
Example 1.}\label{ex1:mesh}
\end{figure}

\begin{figure}
\centering
\includegraphics[width=0.8\textwidth]{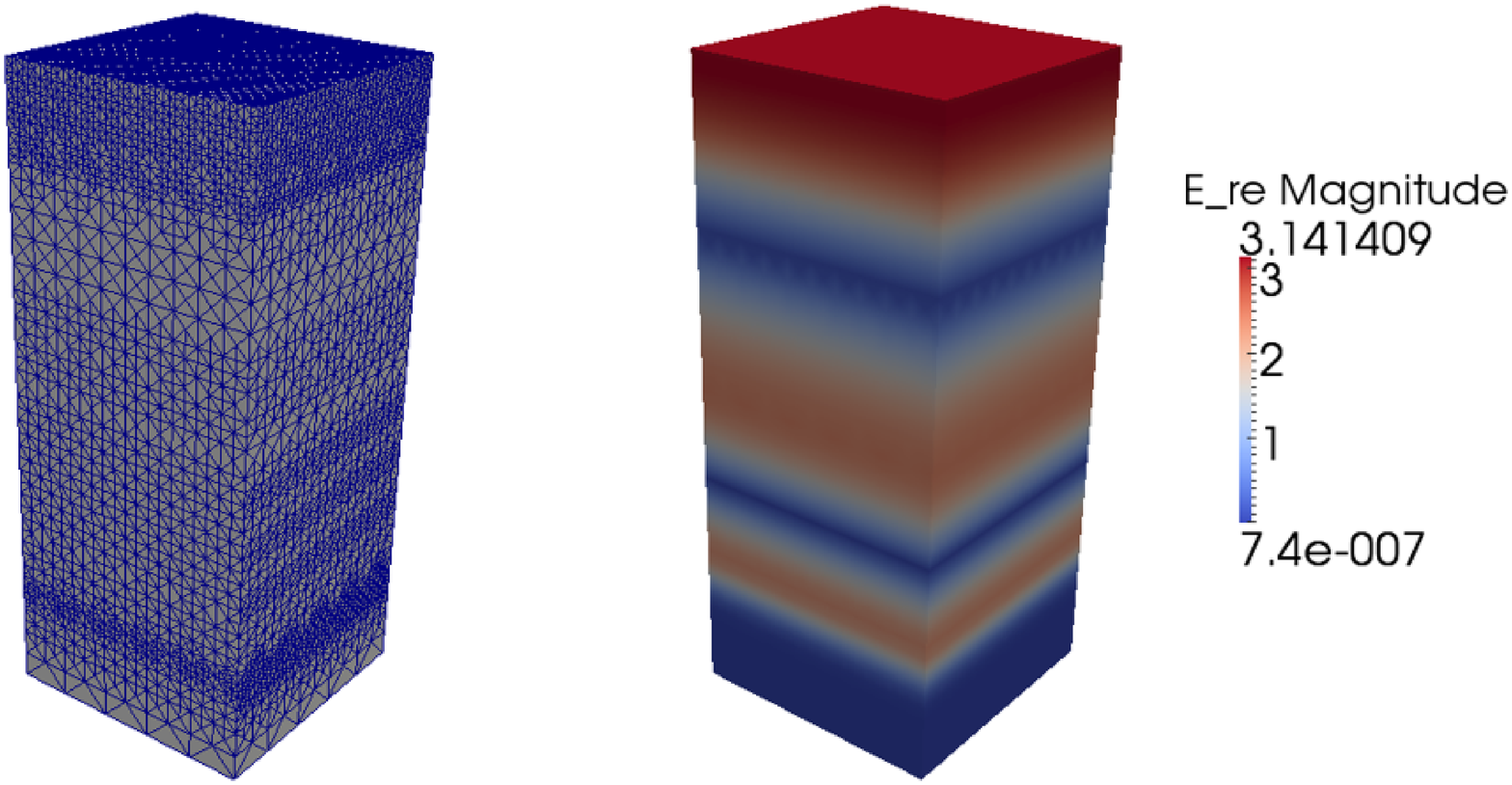}
\caption{The adaptive PML method: The mesh plot and the surface plot of the
amplitude of the field $\boldsymbol E_h$ after 11 adaptive iterations for
Example 1.}\label{ex1pml:mesh}
\end{figure}

\begin{figure}
\centering
\includegraphics[width=0.45\textwidth]{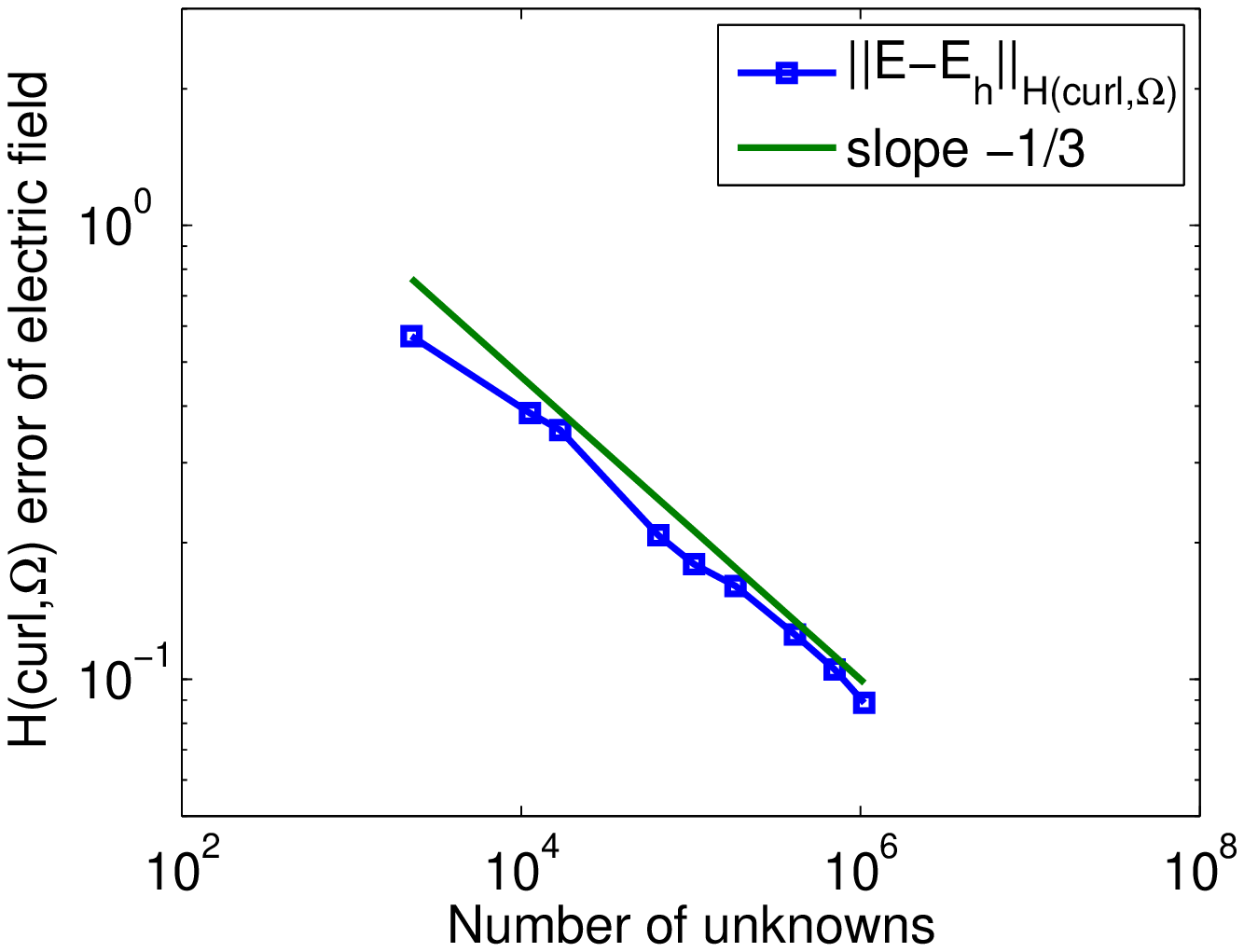}
\includegraphics[width=0.45\textwidth]{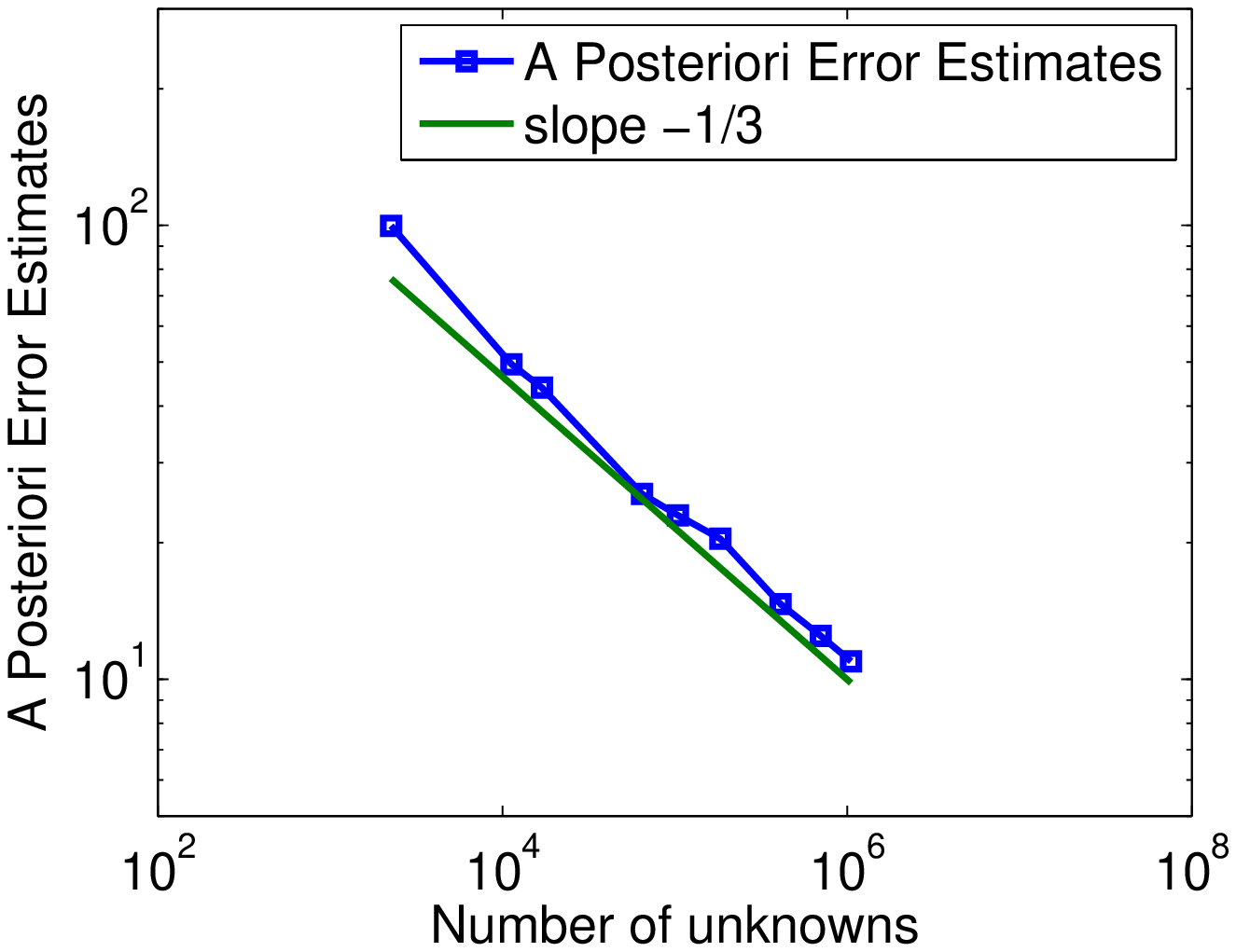}
\caption{The adaptive DtN method: Quasi-optimality of
$\|\bm{E}-\bm{E}_h^N\|_{H({\rm curl},\Omega)}$  (left),
and the a posteriori error estimates (right) for Example 1.}\label{ex1:err}
\end{figure}

\begin{figure}
\centering
\includegraphics[width=0.45\textwidth]{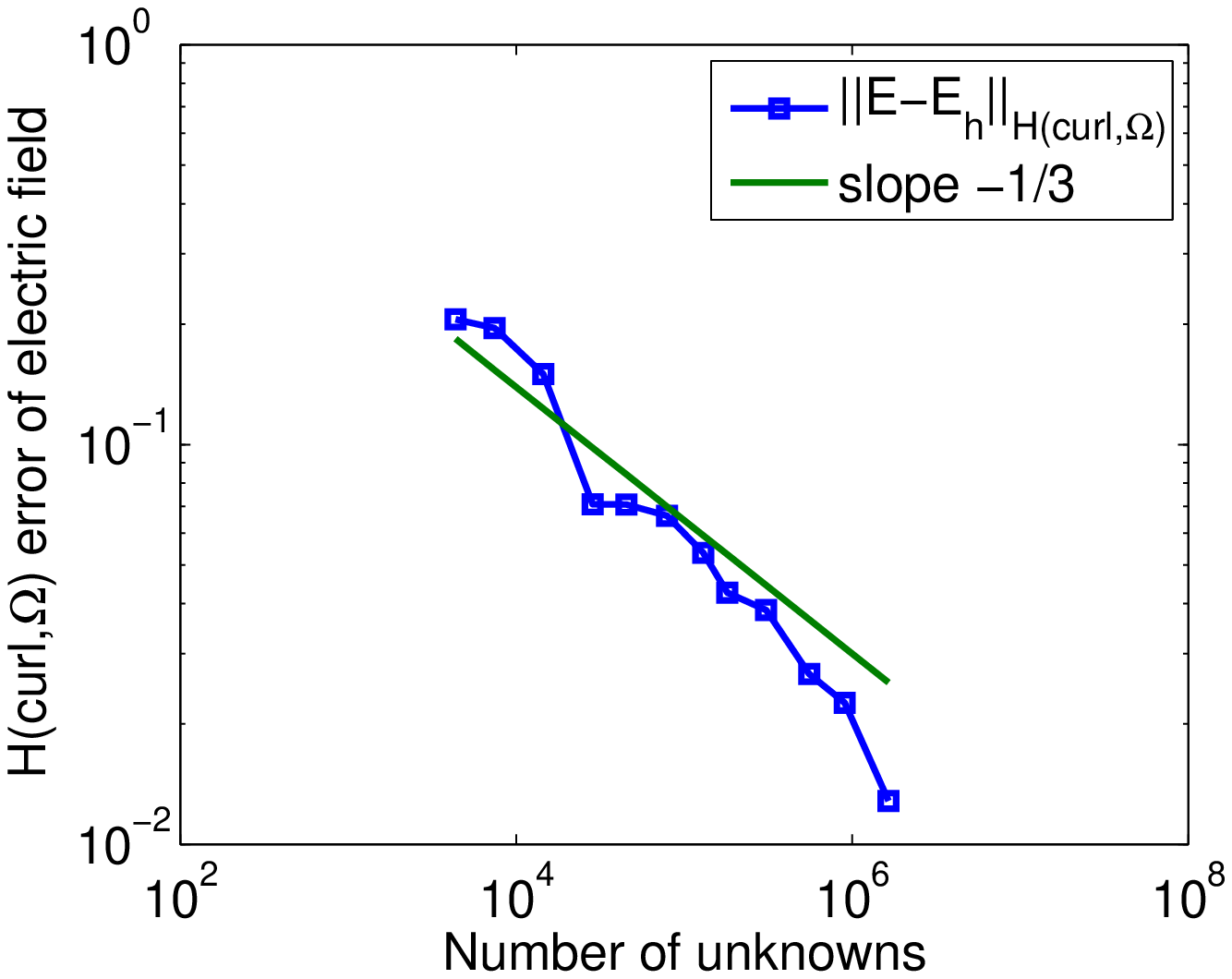}
\includegraphics[width=0.45\textwidth]{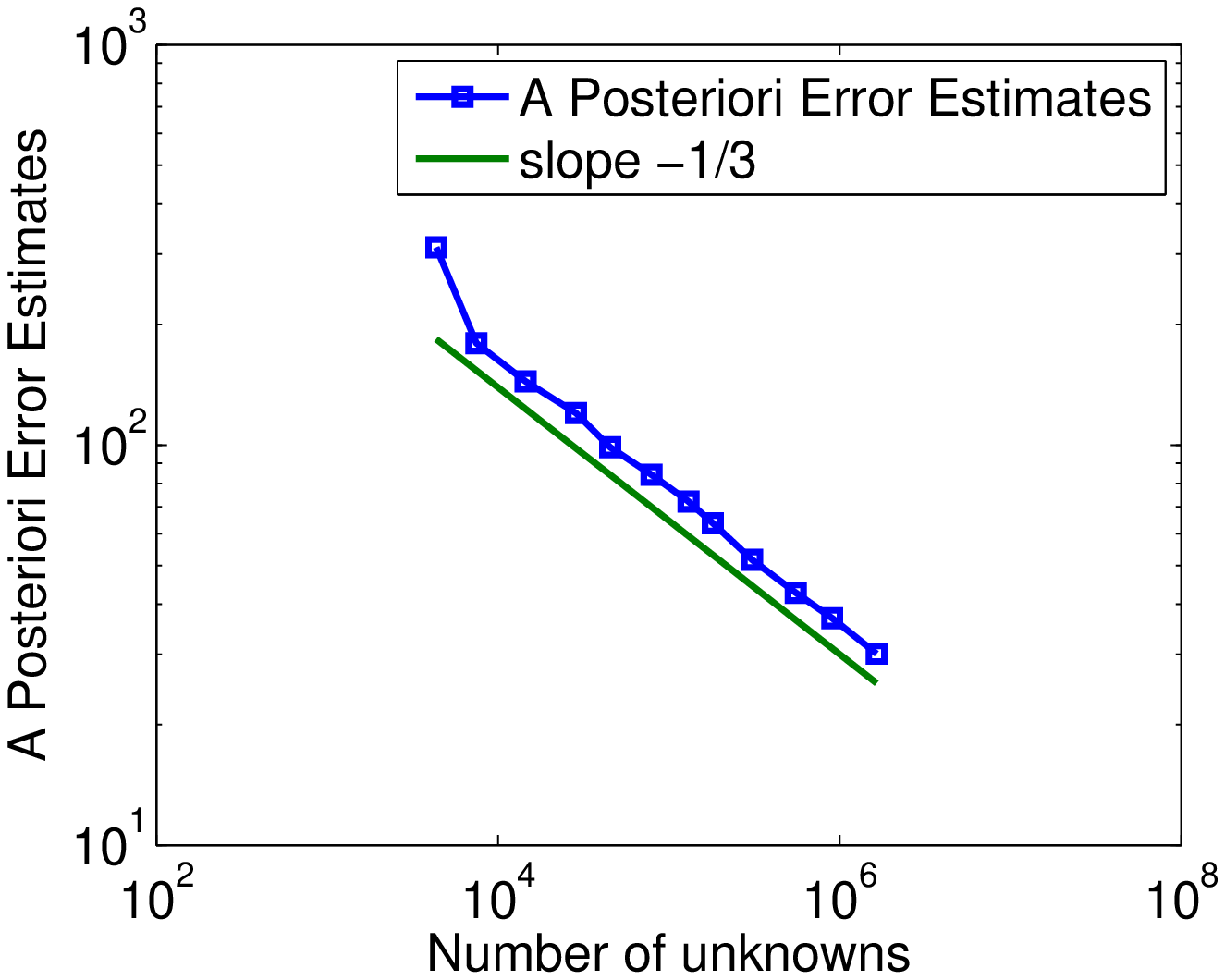}
\caption{The adaptive PML method: Quasi-optimality of
$\|\bm{E}-\bm{E}_h^N\|_{H({\rm curl},\Omega)}$  (left),
and the a posteriori error estimates (right) for Example 1.}\label{ex1pml:err}
\end{figure}

{\em Example 2}. This example concerns the scattering of the time-harmonic
plane wave $\boldsymbol{E}^{\rm inc}$ on the checkerboard grating \cite{L97}, as seen in Figure \ref{ex2:geo}.
The parameters are chosen as $\lambda=1\mu m$, $\theta_1=\theta_2=0$,
$p=(1,1,(\alpha_1+\alpha_2)/\beta)^\top$. The
computational domain is $\Omega=(0,1.25\sqrt{2})\times(0,1.25\sqrt{2})\times(-2,2)$.
Figure \ref{ex2:mesh} shows the mesh and the amplitude of the associated solution for the total field
$\boldsymbol E_h^N$ when the mesh has 1002488 DoFs.
Figure \ref{ex2:est} shows the curves of $\log N_k$ versus the a posteriori error estimates $\eta_h$, and
$\eta_h=O(N^{-1/3}_k)$ is valid asymptotically.

\begin{figure}
\centering
\includegraphics[width=0.36\textwidth]{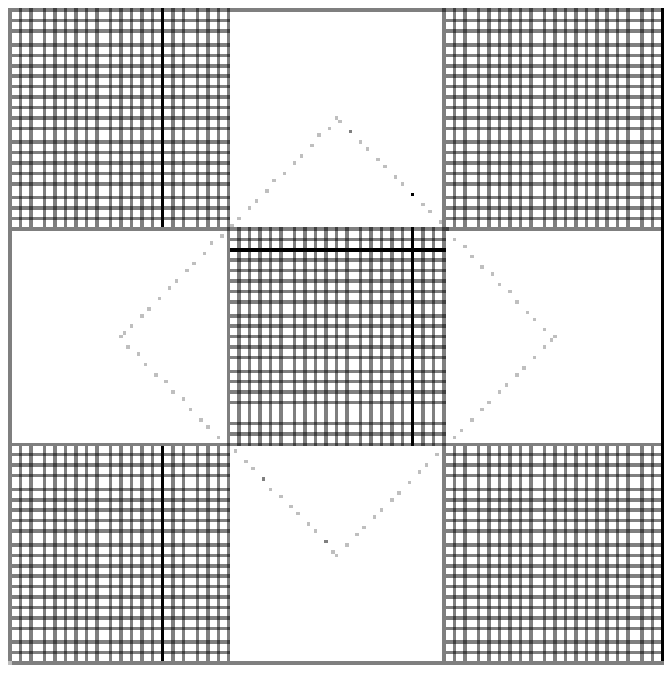}
\hskip 1.cm
\includegraphics[width=0.45\textwidth]{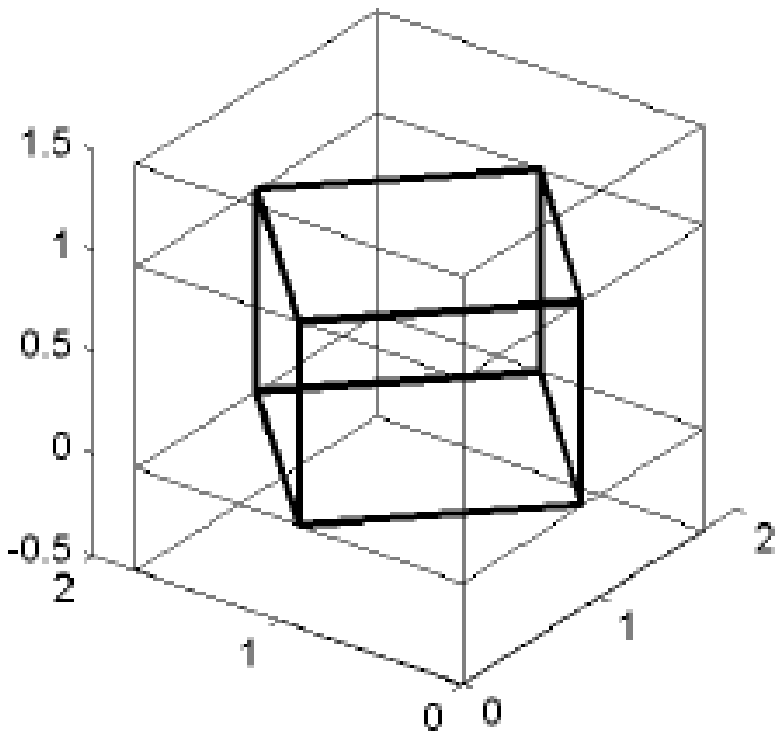}
\caption{A top view of the grating along with the unit cell (left),
and the computational domain (right) for Example 2.}\label{ex2:geo}
\end{figure}

\begin{figure}
\centering
\includegraphics[width=0.8\textwidth]{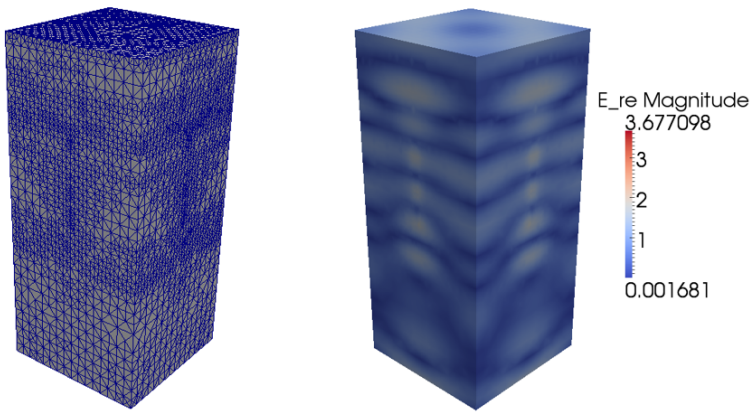}
\caption{The mesh plot and the surface plot of the amplitude
of the field $\boldsymbol E_h$ after 11 adaptive iterations for Example
2.}\label{ex2:mesh}
\end{figure}

\begin{figure}
\centering
\includegraphics[width=0.45\textwidth]{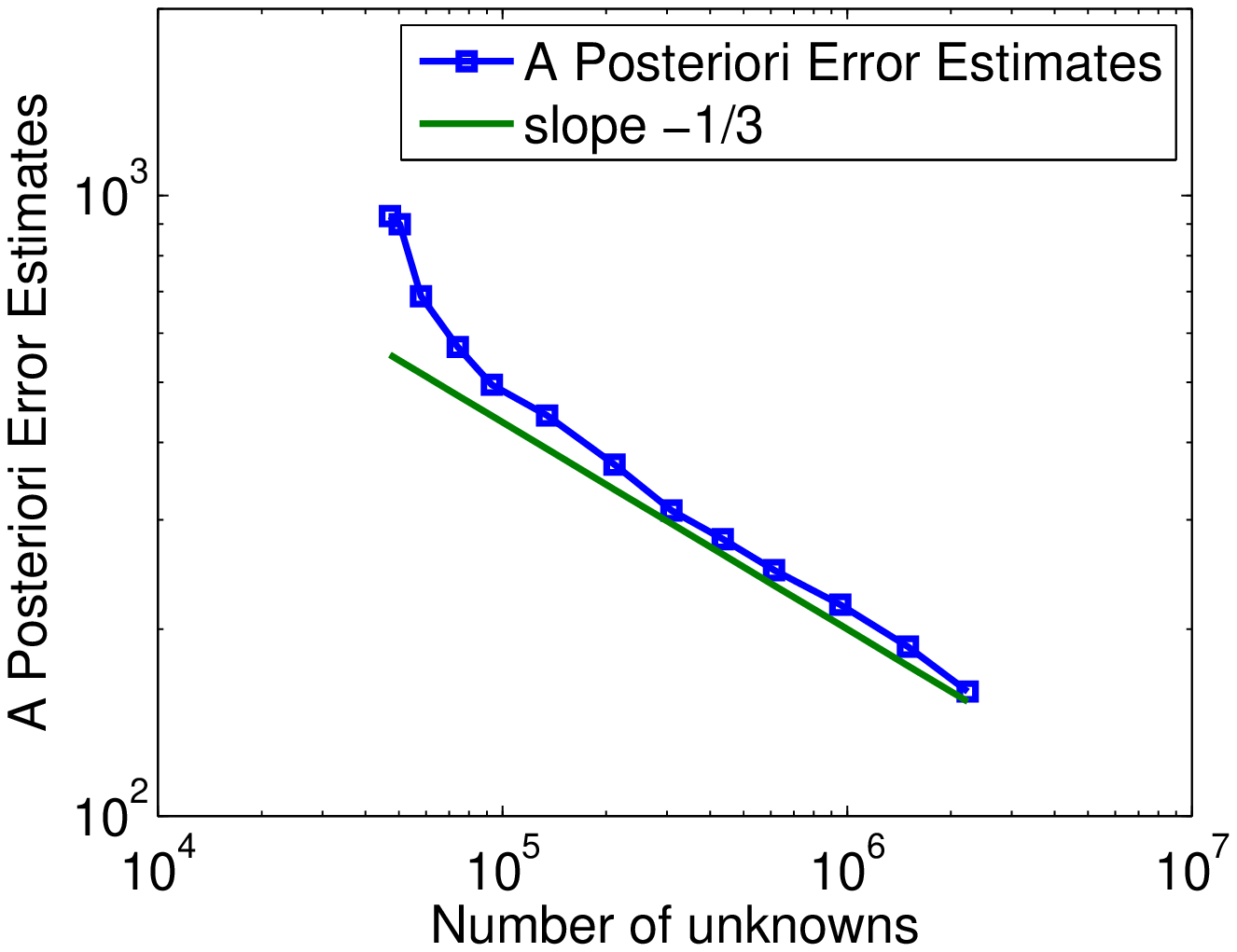}
\includegraphics[width=0.45\textwidth]{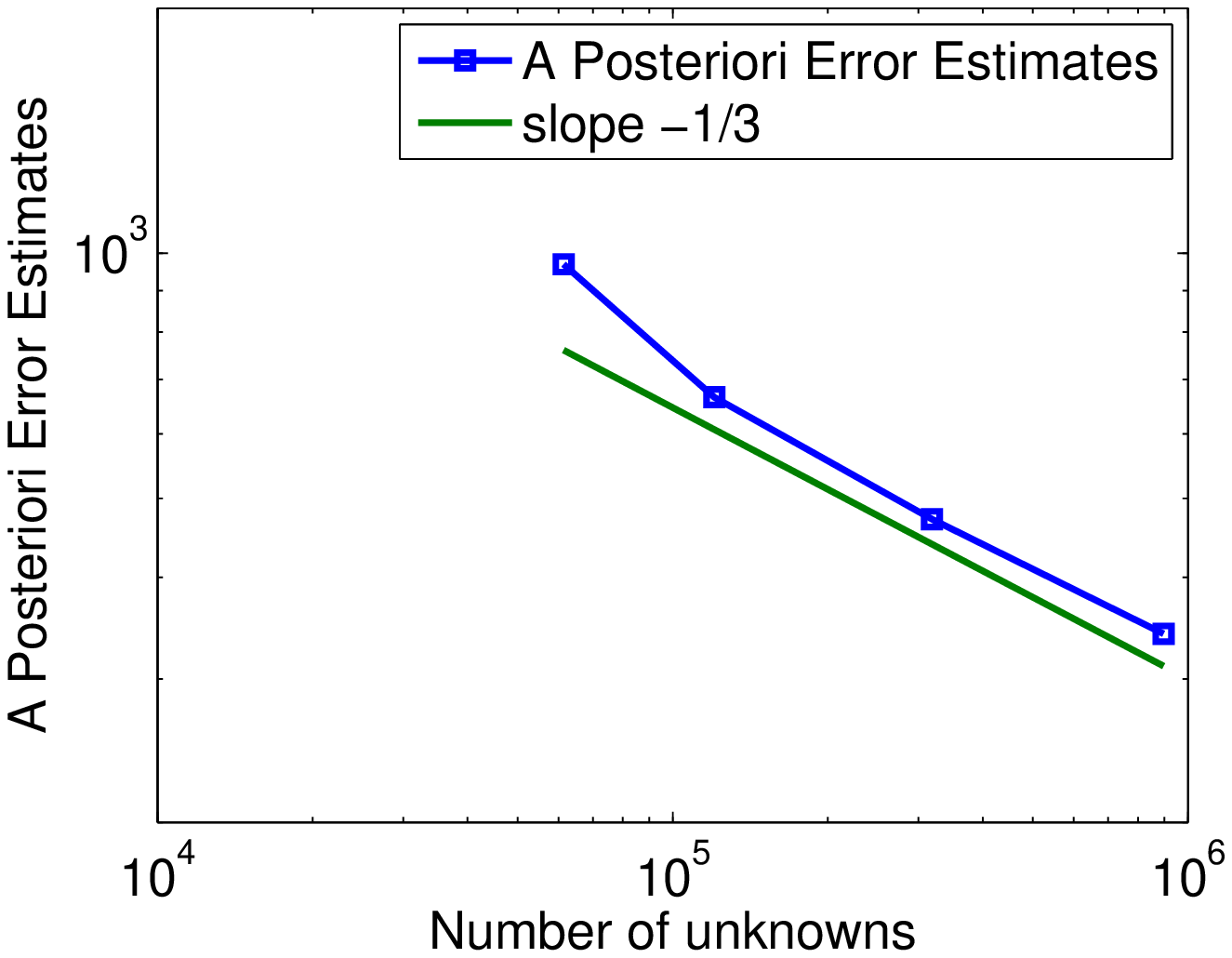}
\caption{Quasi-optimality of the a posteriori error estimates for Example 2.
Left: The adaptive DtN method, right: The adaptive PML method.}\label{ex2:est}
\end{figure}

\section{Concluding remarks}

In this paper, we have presented a new adaptive finite
element method with DtN boundary condition for the diffraction
problem in a biperiodic structure. The a posteriori error estimate takes into
account of the finite element discretization error and the DtN truncation
error, and is used to design the adaptive method to determine the DtN
truncation parameter and choose element for refinements. Numerical results show
that the proposed method is competitive with the adaptive PML method. This work
provides a viable alternative to the adaptive finite element method with PML for
solving the same problem and enriches the range of choices available for solving
many other wave propagation problems. We hope that the method can be applied to
other scientific areas where the problems are proposed in unbounded domains,
especially in the areas where the PML technique might not be applicable.


\begin{thebibliography}{00}

\bibitem{BA73}
{\sc I. Babu\v{s}ka and A. Aziz}, {\em Survey Lectures on Mathematical
Foundations of the Finite Element Method}, in The Mathematical Foundations of
the Finite Element Method with Application to the Partial Differential
Equations, ed. by A. Aziz, Academic Press, New York, 1973, pp. 5--359.

\bibitem{BT80}
{\sc A. Bayliss and E. Turkel}, {\em Radiation boundary conditions for numerical
simulation of waves}, Comm. Pure Appl. Math., 33 (1980), pp. 707--725.

\bibitem{B95}
{\sc G. Bao}, {\em Finite element approximation of time harmonic waves in
periodic structures}, SIAM J. Numer. Anal., 32 (1995), pp. 1155--1169.

\bibitem{B97}
{\sc G. Bao}, {\em Variational approximation of Maxwell's equations in
biperiodic structures}, SIAM J. Appl. Math., 57 (1997), pp. 364--381.

\bibitem{BCW05}
{\sc G. Bao, Z. Chen, and H. Wu}, {\em Adaptive finite element method for
diffraction gratings}, J. Opt. Soc. Amer. A, 22 (2005), pp. 1106--1114.

\bibitem{BCM01}
{\sc G. Bao, L. Cowsar, and W. Masters}, {\em Mathematical Modeling in Optical
Sciences}, Frontiers Appl. Math., vol. 22, SIAM, Philadelphia, 2001.

\bibitem{BCL14}
{\sc G. Bao, T. Cui, and P. Li}, {\em Inverse diffraction grating of
Maxwell's equations in biperiodic structures}, Opt. Express, 22 (2014),
pp. 4799--4816.

\bibitem{BDC95}
{\sc G. Bao, D. C. Dobson, and J. A. Cox}, {\em Mathematical studies in
rigorous grating theory}, J. Opt. Soc. Amer. A, 12 (1995), pp. 1029--1042.

\bibitem{BLW10}
{\sc G. Bao, P. Li, and H. Wu}, {\em An adaptive edge element method with
perfectly matched absorbing layers for wave scattering by periodic structures},
Math. Comp., 79 (2010), pp. 1--34.

\bibitem{BW05}
{\sc G. Bao and H. Wu}, {\em Convergence analysis of the perfectly matched layer
problems for time-harmonic Maxwell's equations}, SIAM J. Numer. Anal., 43
(2005), pp. 2121--2143.

\bibitem{B94}
{\sc J.-P. Berenger}, {\em A perfectly matched layer for the absorption of
electromagnetic waves}, J. Comput. Phys., 114 (1994), pp. 185--200.

\bibitem{BP08}
{\sc J. Bramble and J. Pasciak}, {\em Analysis of a finite elment PML
approximation for the three dimensional time-harmonic Maxwell problem}, Math.
Comp., 77 (2008), pp. 1--10.

\bibitem{BR93a}
{\sc O. Bruno and F. Reitich}, {\em Numerical solution of diffraction problems:
a method of variation of boundaries}, J. Opt. Soc. Am. A, 10 (1993), pp.
1168--1175.

\bibitem{BR93b}
{\sc O. Bruno and F. Reitich}, {\em Numerical solution of diffraction problems:
a method of variation of boundaries. III. Doubly periodic gratings}, J. Opt.
Soc. Am. A, 10 (1993), pp. 2551--2562.



\bibitem{CC08}
{\sc Z. Chen and J. Chen}, {\em An adaptive perfectly matched layer technique
for 3-D time-harmonic electromagnetic scattering problems}, Math. Comp., 77
(2008), pp. 673--698.

\bibitem{CL05}
{\sc Z. Chen and X. Liu}, {\em An adaptive perfectly matched layer technique for
time-harmonic scattering problems}, SIAM J. Numer. Anal., 43 (2005),
pp. 645--671.

\bibitem{CW03}
{\sc Z. Chen and H. Wu}, {\em An adaptive finite element method with perfectly
matched absorbing layers for the wave scattering by periodic structures}, SIAM
J. Numer. Anal., 41 (2003), pp. 799--826.

\bibitem{CM98}
{\sc F. Collino and P. Monk}, {\em The perfectly matched layer in curvilinear
coordinates}, SIAM J. Sci. Comput., 19 (1998), pp. 2061--2090.

\bibitem{CK83}
{\sc D. Colton and R. Kress}, {\em Integral Equation Methods in Scattering
Theory}, John Wiley $\&$ Sons, New York, 1983.

\bibitem{CK98}
{\sc D. Colton and R. Kress}, {\em Inverse Acoustic and Electromagnetic
Scattering Theory}, Second Edition, Springer, Berlin, New York, 1998.

\bibitem{DF92}
{\sc D. C. Dobson and A. Friedman}, {\em The time-harmonic Maxwell equations in
a doubly periodic structure}, J. Math. Anal. Appl., 166 (1992), 507--528.

\bibitem{EM77}
{\sc B. Engquist and A. Majda}, {\em Absorbing boundary conditions for the
numerical simulation of waves}, Math. Comp., 31 (1977), pp. 629--651.

\bibitem{GK95}
{\sc M. Grote and J. Keller}, {\em On nonreflecting boundary conditions}, J.
Comput. Phys., 122 (1995), pp. 231--243.

\bibitem{GK04}
{\sc M. Grote and C. Kirsch}, {\em Dirichlet-to-Neumann boundary conditions for
multiple scattering problems}, J. Comput. Phys., 201 (2004), pp. 630--650.

\bibitem{H99}
{\sc T. Hagstrom}, {\em Radiation boundary conditions for the numerical
simulation of waves}, Acta Numerica (1999), pp. 47--106.

\bibitem{HNS12}
{\sc Y. He, D. P. Nicholls, and J. Shen}, {\em An efficient and stable spectral
method for electromagnetic scattering from a layered periodic struture}, J.
Comput. Phys., 231 (2012), pp. 3007--3022.

\bibitem{HNPX11}
{\sc G. C. Hsiao, N. Nigam, J. E. Pasciak, L. Xu}, {\em Error analysis of the
DtN-FEM for the scattering problem in acoustics via Fourier analysis}, J.
Comput. Appl. Math., 235 (2011), pp. 4949--4965.

\bibitem{JL17}
{\sc X. Jiang, and P. Li}, {\em Inverse electromagnetic diffraction by
biperiodic dielectric gratings}, Inverse Probl., 33 (2017), pp. 085004.

\bibitem{JLLZ17}
{\sc X. Jiang, P. Li, J. Lv, and W. Zheng}, {\em An adaptive finite element method
for the wave scattering with transparent boundary condition}, J. Sci. Comput.,
72 (2017), pp. 936-956.

\bibitem{JLZ13}
{\sc X. Jiang, P. Li, and W. Zheng}, {\em Numerical solution of acoustic
scattering by an adaptive DtN finite element method}, Commun. Comput. Phys., 13
(2013), pp. 1227--1244.

\bibitem{J93}
{\sc J. Jin}, {\em The Finite Element Method in Electromagnetics}, New York:
Wiley, 1993.

\bibitem{L97}
{\sc L. Li}, {\em New formulation of the Fourier modal method for crossed surface-relief gratings}, J.
Opt. Soc. Amer. A 14 (1997), 2758--2767.

\bibitem{LWZ11}
{\sc P. Li, H. Wu, and W. Zheng}, {\em Electromagnetic scattering by unbounded
rough surfaces}, SIAM J. Math. Anal., 43 (2011), 1205--1231.

\bibitem{M03}
{\sc P. Monk}, {\em Finite Element Methods for Maxwell's Equations}, Oxford
University Press, Oxford, UK, 2003.

\bibitem{N01}
{\sc J.-C. N\'{e}d\'{e}lec}, {\em Acoustic and Electromagnetic Equations
Integral Representations for Harmonic Problems}, Springer-Verlag, New
York, 2001.

\bibitem{NS91}
{\sc J.-C. N\'{e}d\'{e}lec and F. Starling}, {\em Integral equation methods in
a quasi-periodic diffraction problem for the time-harmonic Maxwell's
equations}, SIAM J. Math. Anal., 22 (1991), pp. 1679--1701.

\bibitem{P80}
{\sc R. Petit}, {\em Electromagnetic Theory of Gratings}, Topics in Current
Physics, vol. 22, Springer-Verlag, Heidelberg, 1980.

\bibitem{phg}
PHG (Parallel Hierarchical Grid), http://lsec.cc.ac.cn/phg/.

\bibitem{R07}
{\sc P. Rayleigh}, {\em On the dynamical theory of gratings}, R. Soc. London
Ser. A, 79 (1907), pp. 399--416.


\bibitem{TC01}
{\sc F. L. Teixeira and W. C. Chew}, {\em Advances in the theory of perfectly
matched layers}, in Fast and Efficient Algorithms in Computational
Electromagnetics, W. C. Chew et al., eds., Artech House, Boston, 2001,
pp. 283--346.

\bibitem{TY98}
{\sc E. Turkel and A. Yefet}, {\em Absorbing PML boundary layers for wave-like
equations}, Appl. Numer. Math., 27 (1998), pp. 533--557.

\bibitem{WBLLW15}
{\sc Z. Wang, G. Bao, J. Li, P. Li, and H. Wu}, {\em An adaptive finite element
method for the diffraction grating problem with transparent boundary condition},
SIAM J. Numer. Anal., 53 (2015), pp. 1585--1607.

\bibitem{WL09}
{\sc Y. Wu and Y. Y. Lu}, {\em Analyzing diffraction gratings by a boundary
integral equation Neumann-to-Dirichlet map method}, J. Opt. Soc. Am. A, 26
(2009), pp. 2444--2451.


\end{thebibliography}
\end{document}